\documentclass[a4paper,reqno]{amsart}
\usepackage{graphicx}
\usepackage{amssymb,amsmath,amsthm, amsfonts}
\usepackage[latin1]{inputenc}
\usepackage[english]{babel} 
\usepackage{enumitem}
\usepackage[affil-it]{authblk}

\newtheorem{theorem}{Theorem}
\newtheorem{definition}{Definition}
\newtheorem{proposition}{Proposition}
\newtheorem{remark}{Remark}
\newtheorem{corollary}{Corollary}
\newtheorem{lemma}{Lemma}
\newtheorem{acknowledgements}{Acknowledgements}

\newcommand{\R}{\mathbb{R}}

\newcommand{\Z}{\mathbb{Z}}
\newcommand{\rhogamma}{{\rho}_{\Gamma}}
\newcommand{\rgamma}{{r}_{\Gamma}}
\newcommand{\dd}{\,d}

\newcommand{\Lip}{\operatorname{Lip}}
\newcommand{\Id}{\operatorname{Id}}
\newcommand{\id}{\operatorname{Id}}

\newcommand{\spec}{\operatorname{Spec}}
\newcommand{\krn}{\operatorname{Ker}}
\newcommand{\wt}[1]{\widetilde{#1}}

\begin{document}

\title{Normal forms and Sternberg conjugation theorems for infinite dimensional coupled map lattices}

\title{Sternberg theorems for coupled map lattices}

\author[Ruben~Berenguel and Ernest~Fontich]{}
\subjclass{37L60 \and 37C15 \and 37G05}
\keywords{Dynamical systems on lattices \and Linearisation \and Normal form}
\email{ruben@mostlymaths.net}
\email{fontich@ub.edu}

\maketitle

\centerline{\scshape Ruben Berenguel}
\medskip
{\footnotesize
 \centerline{Independent Researcher,}
}
\medskip

\centerline{\scshape Ernest Fontich}
{\footnotesize
 \centerline{Departament de Matem\`{a}tiques i Inform\`atica,}
 \centerline{Institut de Matem\`atiques de la Universitat de Barcelona (IMUB),}
 \centerline{Barcelona Graduate School of Mathematics (BGSMath),}
 \centerline{Universitat de Barcelona (UB),}
 \centerline{Gran Via 585,
 08007 Barcelona, Spain} }

\bigskip

\begin{abstract}
  In this paper we present local Sternberg conjugation theorems near attracting
  fixed points for lattice systems. The interactions are spatially decaying and
  are not restricted to finite distance. The conjugations obtained 
  retain the same spatial decay. In the presence of resonances the conjugations
  are to a polynomial normal form that also has decaying properties.
\end{abstract}


\section{Introduction}

Coupled map lattices are used to model many systems in physics, chemistry and
biology. They are formed by sequences of nodes, each one having its own internal
dynamics and being influenced by the dynamics of other nodes through some
interactions.

Its origin can be found in the first models for the dynamics of chains
of particles under the action of a potential, with a nearest neighbours
interaction, models which were first considered by Prandtl \cite{Prandtl1928}
and Dehlinger \cite{Dehlinger1929}. Later, these models were also considered by
Frenkel and Kontorova for specific cases in dislocation models of solids in
\cite{Kontorova1938_1} and~\cite{Frenkel1939}.

Several problems can be studied under the Frenkel-Kontorova model (or some
generalization), ranging from chains of coupled pendula, dislocation dynamics
and surface physics to DNA and neural dynamics. See \cite{FrenkelKontorova} for
a modern description and many applications of this model.

In statistical mechanics, coupled oscillations were used to study numerically
the equirepartion of energy, starting with \cite{FermiPU}. See also
\cite{GallavottiFPU} for a modern treatment of the problem. Mathematically, they
appear as models of discretised partial differential equations. Several objects
and notions are studied in this setting, such as travelling waves, wave fronts,
invariant measures and spatio-temporal chaos, see \cite{kaneko1993theory},
\cite{B-S88}, \cite {Afra05}, \cite{FP99}, \cite{dauxois2002dynamics},
\cite{floria2005frenkel}, \cite{jiang-pesin}, \cite{Pesin-Yurchenko}. For
applications to neuroscience and biology see \cite{Ermentrout2010},
\cite{Izhikevich07}, \cite{peyrard2004}, \cite{peyrard2004breathers}.

One can consider higher dimensional lattices with interactions among all
particles. In this case, we have to require some decay in the strength of the
interaction, because as it is physically natural, the larger the separation
between particles is, the smaller the force of interaction should be.

Assuming each node is represented by $\R^n$, in this paper we consider
$m$-dimensional lattices modelled as
$$
\ell^\infty = \ell^\infty(\R^n) = \{x:\Z^m\to \R^n \mid \sup_{k\in \Z^m} |x(k)| < \infty \}.
$$
We allow each node to interact with every other node, but the strength of the
interactions decay with the distance between them with a spatial decay
controlled by a function $\Gamma:\Z^m \to [0,\infty)$ satisfying certain
properties. To this end, we will use the decay functions introduced in
\cite{JiangDeLaLlaveSRBMeasures} and presented in Section
\ref{Section:DecayFuns}. We will work with differentiable maps $F: \ell^\infty
\to \ell^\infty $ such that the derivative of the component of $F$ corresponding
to the $i$-th node with respect to the variable $x_j$ of the $j$-th node
satisfies
$$
\Big|\frac{\partial F_i}{\partial x_j} \Big| \le C\Gamma(i-j).
$$

More generally, we will work with spaces $C^r_\Gamma(\ell^\infty,\ell^\infty )$
of $C^r $ functions having decay properties. To be able to work with them, first
we have to introduce linear and multilinear maps with decay in $\ell^\infty $
spaces. See Sections \ref{Section:LinearAndMultilinear} and
\ref{Section:DiffFunDecay} for the precise definitions and properties. These
spaces or similar constructs were introduced in \cite{JiangDeLaLlaveSRBMeasures},
\cite{FontichDelaLlaveMartinFAFramework}, \cite{FontichDelaLlaveMartinHSIM},
\cite{FontichDeLaLlaveSireWhiskeredLattices15}.

In this paper we describe normal forms and Sternberg theorems \cite{SternbergI}
around fixed points in the setting described above, which give differentiable
conjugations of the map to their normal forms in neighbourhoods of attracting
fixed points. In absence of resonances, the normal forms reduce to the
linearisation of the map at the fixed point. For more general fixed points, even
in finite dimensions, the study seems to require the use of differentiable bump
functions in the ambient space, as it is the case in other settings we are aware
of \cite{stern58}, \cite{BanyagaLlaveWayne}, \cite{IlyashenkoYakovenko},
\cite{ChaperonCoudrayInvariantManifolds}. However, such bump functions do not
exist in $\ell^\infty(\R^n)$.

One important consequence of our results is that the normal forms and the
obtained conjugations have the same kind of decay as the original map.

From the differentiable conjugation to the linear map we can obtain several
invariant manifolds: if we can linearise the map, we can find as many manifolds
as linear invariant subspaces the linear map has. Among them, the slow
manifolds. These define the motion which converges to the attractor the slowest,
and contain the dynamics that can be observed in simulations or physical
systems. These manifolds have parameterisations that decay in the aforementioned
sense. This property is important in the study of statistical mechanics, see
\cite{JiangDeLaLlaveSRBMeasures}.

In the study of normal forms and the linearisation procedure we have to deal
with cohomological equations in $\ell^\infty(\R^n)$, in the setting of linear
maps with decay. For this, we use the so-called Sylvester operators (see
\cite{CabreFontichDeLaLlavePMIM1}) and we adapt the theory to work in the space
of $k$-linear maps with decay and study their invertibility properties. To that
end we introduce the $\Gamma$-spectrum in Section
\ref{Preliminaries:Functional}, a tool enabling us to study these operators in
this setting.


We obtain Sternberg theorems for the conjugation of a map to its
linear part or to its normal form, in the case that the linear part is
a contraction (Poincar\'e domain). Assuming decay properties for the
map we obtain decay properties for the conjugating map. For the
results where we allow the existence of resonances, we use a normal
form theory with decay which we develop here (analogous to the
standard normal form theory around a fixed point) and based on the
use of Sylvester operators in spaces of $k$-linear maps in
$\ell^\infty(\mathbb{R}^n)$, introduced in Section
\ref{Preliminaries:SylvesterInLk}.

We study two cases, the first one for maps that are small perturbations of an
uncoupled map with equal dynamics in each node. For this class of maps we add
conditions on the eigenvalues of the linearisation of the unperturbed map at the
fixed point restricted to a node (all maps are the same in each node).

In the absence of resonances among eigenvalues we have the following result that
gives differentiable conjugation to the linear part of the map. The norm
$\|\cdot\|_\Gamma$ is introduced in Section \ref{Section:LinearAndMultilinear}
and the space of $C^r$ functions with decay $C^r_\Gamma$ is introduced in
Section \ref{Section:DiffFunDecay}.

\begin{theorem}\label{Sternberg:Theorem:Sternberg}
	Let $U$ be an open set of $\ell^\infty(\mathbb{R}^n)$ such that
	$0\in U$. Let $F:U\to\ell^\infty(\mathbb{R}^n)$ be a $C^r_{\Gamma}$
	map of the form $F=F_0+F_1$ where $F_0$ is an uncoupled map and
	$F_0(0)=F_1(0)=0$. Let $A=DF_0(0)$, $B=DF_1(0)$ and $M=A+B$. Assume
	that $A_{ij}=\mathfrak{a}\delta_{ij}$ with $\mathfrak{a}\in
	L(\mathbb{R}^n, \mathbb{R}^n)$.
	
	Let $\spec(\mathfrak{a})=\{\lambda_1,\ldots,\lambda_n\}$,
	$\alpha=\min_{i}|\lambda_i|$, $\beta=\max_{i}|\lambda_i|$,
	$\nu=\frac{\log\alpha}{\log\beta}$ and $r_0=[\nu]+1$. Assume
	\begin{enumerate}[label=(H\arabic*)]
		\item $0< |\lambda_i|<1,\quad 1\leq i\leq
		n$, \label{Sternberg:Theorem:Sternberg:H1}
		\item $\lambda_i\neq \lambda_1^{k_1}\cdots\lambda_n^{k_n},\quad k=(k_1, \ldots, k_n)\in (\mathbb{Z}^+)^n$, $2\leq |k|\leq r_0$, $1\leq i\leq n$, where $|k| = k_1+\ldots+k_n$.\label{Sternberg:Theorem:Sternberg:H2}
	\end{enumerate}
	Then, if $F\in
	C^r_\Gamma(U, \ell^\infty(\mathbb{R}^n))$ with $r\geq r_0$ and
	$\|B\|_\Gamma$ is small enough, there exists $R\in
	C^r_\Gamma(\ell^\infty(\mathbb{R}^n), \ell^\infty(\mathbb{R}^n))$ such
	that $R(0)=0$, $DR(0)=\id$ and
	\[R\circ F=MR\] in some neighborhood $U_1\subseteq U$ of $0$ in
	$\ell^\infty(\mathbb{R}^n)$.
\end{theorem}

If we allow for the existence of resonances, i.e. omitting Hypothesis
\ref{Sternberg:Theorem:Sternberg:H2} above, we have an analogous result giving a
differentiable conjugation to a polynomial normal form (Theorem
\ref{Sternberg:Theorem:SternbergNR} in Section \ref{sec:sternberg-thms}).

The second case we consider is non-perturbative and requires conditions over the
$\Gamma$-spectrum of the linear part, introduced in Section
\ref{Preliminaries:Functional}.

\begin{theorem}\label{Sternberg:Theorem:SternbergSpectral}
	Let $U$ be an open set of $\ell^\infty(\mathbb{R}^n)$ such that
	$0\in U$. Let $F\in C^r_\Gamma(U, \ell^\infty(\mathbb{R}^n))$ with
	$F(0)=0$. Let $A=DF(0)$, $\alpha_{\Gamma}=\inf\{|\lambda|\,\vert\, \lambda\in\spec_\Gamma(A)\}$,
	$\beta_{\Gamma}=\sup\{|\lambda|\,\vert\, \lambda\in\spec_\Gamma(A)\}$,
	$\nu=\frac{\log\alpha_{\Gamma}}{\log\beta_{\Gamma}}$ and $r_0=[\nu]+1$. Assume
	\begin{enumerate}[label=(H\arabic*)]
		\item $0\notin\spec_\Gamma(A)\;\;\;$ and $\;\;\;\spec_\Gamma(A)\subset
		\mathbb{D}(0,1)$,\label{Sternberg:Theorem:SternbergSpectral:H1}
		\item $\spec_\Gamma(A)\cap\left(\spec_\Gamma(A)\right)^j=\emptyset$,\qquad
		$2\leq j \leq r_0$.\label{Sternberg:Theorem:SternbergSpectral:H2}
	\end{enumerate}
	Then, if $F\in
	C^r_\Gamma(U, \ell^\infty(\mathbb{R}^n))$ with $r\geq r_0$ there
	exists $R\in C^r_\Gamma(\ell^\infty(\mathbb{R}^n),
	\ell^\infty(\mathbb{R}^n))$ such that $R(0)=0$, $DR(0)=\id$ and
	\[R\circ F=AR\] in a neighborhood $U_1\subset U$ of 0.
\end{theorem}

The paper is structured as follows. Section \ref{Section:DecayFuns} provides the
definition of the class of decay functions we will work with and the main
examples of them. Section \ref{Section:LinearAndMultilinear} deals with linear
and multilinear maps with decay while Section \ref{Section:DiffFunDecay} deals
with $C^k$ maps with decay. In Section \ref{Preliminaries:Functional} we
introduce the $\Gamma$-spectrum of linear maps with decay. In Section
\ref{Preliminaries:SylvesterInLk} we study some spectral properties of Sylvester
operators. Finally, Sections \ref{Section:NormalFormsSternberg} and
\ref{sec:sternberg-thms} provide the normal forms and the conjugation
results respectively.



\section{Lattices, decay functions and dynamical systems}
\label{Section:DecayFuns}

In this work we will consider dynamical systems in the space of bounded
sequences of points of $\mathbb{R}^n$ with indices in $\mathbb{Z}^m$. That is,
we will work in the infinite product space $(\mathbb{R}^n)^{\mathbb{Z}^m}$,
where as usual we will call {\it node} each individual $\mathbb{R}^n$ in the
lattice. Associated with this space we will consider a decay function, which
will control the strength of the interactions between different nodes.

The space of bounded sequences in the infinite product space
$(\mathbb{R}^n)^{\mathbb{Z}^m}$ is denoted by
$\ell^\infty(\mathbb{R}^n)$ and formally defined as
\begin{equation*}
  \ell^\infty(\mathbb{R}^n)=\left\{
    (x_i)_{i\in\mathbb{Z}^m}\,\vert\, x_i\in\mathbb{R}^n,\, \sup_{i\in\mathbb{Z}^m}\|x_i\|<\infty
  \right\},
\end{equation*}
where $\|\cdot\|$ is a given norm in $\mathbb{R}^n$. We endow
$\ell^\infty(\mathbb{R}^n)$ with the norm
$\|x\|_\infty=\sup_{i\in\mathbb{Z}^m}\|x_i\|$ as usual. Note that if we change
the norm in $\mathbb{R}^n$ we end up with an equivalent norm in
$\ell^\infty(\mathbb{R}^n)$. We denote by
$\mathrm{proj}_i:\ell^\infty(\mathbb{R}^{n})\to\mathbb{R}^n$ the projection onto
the $i$-th component, and the related function,
$\mathrm{emb}_i:\mathbb{R}^n\to\ell^\infty(\mathbb{R}^{n})$ the $i$-th
embedding. They satisfy $\mathrm{proj}_j(\mathrm{emb}_i(u))=0,\,i\neq j$, and
$\mathrm{proj}_i(\mathrm{emb}_i(u))=u$ for every $u\in\mathbb{R}^n$. This
embedding is an isometry if the norm in $\ell^\infty(\mathbb{R}^{n})$ is induced
by the norm considered in $\mathbb{R}^n.$

\subsection{Decay functions in lattices}

\label{Section:DecayFuns}

To be able to define meaningful localised perturbations in
$\ell^\infty(\mathbb{R}^n)$, we consider an appropriate set of weighted Banach
spaces. The main idea is that the coupling term in the perturbed system belongs
to a weighted space, which controls the strength of the interaction between
nodes. We should note that nearest-neighbour coupling (or any other finite rank
coupling) will satisfy these hypotheses.

We will make use of the following decay functions, originally
introduced in \cite{JiangDeLaLlaveSRBMeasures}.

\begin{definition}
  \label{DecayFuns:Definition:DecayFunction} We say that a function
  $\Gamma:\mathbb{Z}^m\to\mathbb{R}^+$ is a decay function when it
  satisfies:
  \begin{enumerate}
  \item $\sum_{k\in \mathbb{Z}^m}\Gamma(k)\leq 1$,
  \item
    $\sum_{k\in\mathbb{Z}^m}\Gamma(i-k)\Gamma(k-j)\leq\Gamma(i-j),\qquad
    i,\,j\in\mathbb{Z}^m.$
  \end{enumerate}
\end{definition}

The first property ensures that interaction propagation related to such a decay
function is finite, while the second property is akin to a triangular inequality
in a discrete lattice. As pointed out by Prof. L. Sadun, the second property can
be interpreted as that the sum of the interactions between two nodes through the
interactions involving third nodes is dominated by the direct interaction
between them.

The following proposition can be found in
\cite{JiangDeLaLlaveSRBMeasures} and provides a family of examples of decay
functions satisfying Definition
\ref{DecayFuns:Definition:DecayFunction}.

\begin{proposition}
  Given $\alpha>m,\theta\geq0$, there exists $a>0$, depending on
  $\alpha,\theta, m$ such that the function defined by
  \[\Gamma(j)=\begin{cases}
    a,\quad &j=0,\\
    a|j|^{-\alpha}e^{-\theta|j|},&j\neq 0,
  \end{cases}\]
  is a decay function on $\mathbb{Z}^m$.
\end{proposition}

Note that the standard exponential function
$\Gamma(j)=Ce^{-\theta|j|}$ is not a decay function for any $C,
\theta>0$, as proved in \cite{JiangDeLaLlaveSRBMeasures}.



\section{Linear and multilinear maps with decay}
\label{Section:LinearAndMultilinear}
To define $C^r$ maps with decay properties in lattices we need to first
introduce spaces of linear and multilinear mappings with suitable decay
properties. Then we can use these definitions to introduce spaces of $C^r$ maps
with these predefined decay properties for its derivatives. In this section we
will define linear maps with decay and its related norm $\|\cdot\|_\Gamma$. From
now on we will use $\|\cdot\|$ to denote the norm induced in the space of linear
or multilinear maps by the same norm in $\ell^\infty(\mathbb{R}^n)$. All decay
functions will satisfy Definition \ref{DecayFuns:Definition:DecayFunction}. We
reproduce some statements from \cite{JiangDeLaLlaveSRBMeasures} and
\cite{FontichDelaLlaveMartinFAFramework}, and provide some details and some
additional results for the convenience of the reader.

\subsection{The space of linear maps with decay}
\label{Section:LinearAndMultilinearWithDecay}
The most natural way to define linear maps with decay is to require the
components of ``infinite matrices'' to have decay properties with respect to
their indices. This is formalised as follows. Given $\|\cdot\|$ a norm in
$\mathbb{R}^n$ we define
\[
L_\Gamma=L_\Gamma(\ell^\infty(\mathbb{R}^n),\ell^\infty(\mathbb{R}^n))=\big\{A\in
L(\ell^\infty(\mathbb{R}^n),\ell^\infty(\mathbb{R}^n))\,\vert\,
\|A\|_\Gamma<\infty\big\},
\]
where
\[
\|A\|_\Gamma=\max\{\|A\|,\gamma(A)\}, 
\]
with $\|A\|$ the operator norm of $A$ and
\[
\gamma(A)=\sup_{i,k\in\mathbb{Z}^m}\sup_{\substack{\|u\|\leq1,\\ \mathrm{proj}_ju=0,\;j\neq
    k}}\|(Au)_{i}\|\Gamma(i-k)^{-1}.
\]

\begin{remark}
  We will use $L_\Gamma$ as an abbreviation for
  $L_\Gamma(\ell^\infty(\mathbb{R}^n),\ell^\infty(\mathbb{R}^n))$. Although the
  definition has been written for $\ell^\infty(\mathbb{R}^n)$, we can define
  linear maps with decay among arbitrary vector subspaces $\mathcal{E},
  \mathcal{F}$ of $\ell^\infty(\mathbb{R}^n)$ as
  \[L_\Gamma=L_\Gamma(\mathcal{E},\mathcal{F})=\big\{A\in
  L(\mathcal{E},\mathcal{F})\,\vert\, \|A\|_\Gamma<\infty\big\}.
  \]
  All results stated for
  $L_\Gamma(\ell^\infty(\mathbb{R}^n),\ell^\infty(\mathbb{R}^n))$ extend in a
  straightforward way to $L_\Gamma(\mathcal{E}, \mathcal{F})$.
\end{remark}

\nocite{JiangDeLaLlaveSRBMeasures}
We can provide an interpretation of $\gamma(A)$ in terms of the elements of the 
infinite dimensional matrix. If we denote
$A_{ij}=\mathrm{proj}_iA\,\mathrm{emb}_j$ then we have
$$
\gamma(A)=\sup_{i,j\in\mathbb{Z}^m}\|A_{ij}\|\Gamma(i-j)^{-1}.
$$

It is worth emphasizing that the elements $A_{ij}$ do not determine $A$. We can
find a specific counter-example in \cite[p.2843]{FontichDelaLlaveMartinFAFramework}

\begin{remark}
  \label{LandM:Remark:UncoupledGamma0}
  Observe that given an uncoupled linear map
  $A_{ij}=\mathfrak{a}\delta_{ij}$, with $\mathfrak{a}\in L(\mathbb{R}^n,
  \mathbb{R}^n)$ we have $A\in
  L_\Gamma(\ell^\infty(\mathbb{R}^n),\ell^\infty(\mathbb{R}^n))$,
  $\gamma(A)=\Gamma(0)^{-1}\|\mathfrak{a}\|$ and
  $\|A\|_\Gamma=\Gamma(0)^{-1}\|\mathfrak{a}\|$.
\end{remark}

Several basic properties follow.

\begin{proposition}\label{LandM:Proposition:LGBanachSpace}
  The space $(L_\Gamma,\,\|\cdot\|_\Gamma)$ is a Banach space.
\end{proposition}

\begin{proposition}[Algebra properties]\label{LandM:Proposition:LGAlgebraProperty}
  Let $A,\,B \in L_{\Gamma}(\ell^\infty(\mathbb{R}^n),
  \ell^\infty(\mathbb{R}^n))$. Then $AB\in
  L_{\Gamma}(\ell^\infty(\mathbb{R}^n),\ell^\infty(\mathbb{R}^n))$ and
  \begin{itemize}
  \item[(1)] $\gamma(AB)\leq \gamma(A)\gamma(B)$,
  \item[(2)] $\|AB\|_{\Gamma}\leq \|A\|_{\Gamma}\|B\|_{\Gamma}$.
  \end{itemize}
\end{proposition}

\begin{remark}
  Proposition \ref{LandM:Proposition:LGBanachSpace} and Proposition
  \ref{LandM:Proposition:LGAlgebraProperty} imply
  $L_\Gamma(\ell^\infty(\mathbb{R}^n),\ell^\infty(\mathbb{R}^n))$ is a Banach
  algebra. It has a unit element $\id$ but $\|\id\|_\Gamma\neq 1$. This makes
  spectral theory in $L_\Gamma$ less straightforward, since the classic results
  (cf. \cite{RudinFA}) require unit elements with norm 1 in most proofs. There is however
  a standard trick to overcome this difficulty, see Section \ref{Preliminaries:Functional}.
\end{remark}

\begin{proposition}\label{LandM:Lemma:DInvertibleEasyCase}
  Let $M_0\in L_\Gamma(\ell^\infty(\mathbb{R}^n), \ell^\infty(\mathbb{R}^n))$
  invertible such that \[M_0^{-1}\in L_\Gamma(\ell^\infty(\mathbb{R}^n),
    \ell^\infty(\mathbb{R}^n))\] and $M_1\in L_\Gamma(\ell^\infty(\mathbb{R}^n),
  \ell^\infty(\mathbb{R}^n))$ such that $\|M_0^{-1}\|_\Gamma\|M_1\|_\Gamma<1$.
  Then $M=M_0+M_1$ is invertible, $M^{-1}\in L_\Gamma(\ell^\infty(\mathbb{R}^n),
  \ell^\infty(\mathbb{R}^n))$ and
   \begin{equation}
     |\|M^{-1}\|_\Gamma-\|M_0^{-1}\|_\Gamma|\leq \|M^{-1}-M_0^{-1}\|_\Gamma= \mathcal{O}(\|M_1\|_\Gamma).\nonumber
\end{equation}

\end{proposition}
\begin{proof}
  Since $M_0$ is invertible, we can write
  \[M=M_0\big(\Id + M^{-1}_0{M}_1\big).\] Since
  $\|M_0^{-1}\|_{\Gamma}\|M_1\|_{\Gamma}<1$ we can write $M^{-1}$ as a Neumann series
  as
  \begin{equation}
    M^{-1}=\sum_{j=0}^{\infty}(-M^{-1}_0M_1)^jM_0^{-1}=M_0^{-1}+\sum_{j=1}^\infty(-M_0^{-1}M_1)^jM_0^{-1}\nonumber
  \end{equation}
  which is convergent in $\|\cdot\|_{\Gamma}$.
\end{proof}

\subsection{The space of $k$-linear maps with decay}

To characterise higher order differentiable functions with decay we
inductively define multilinear maps with decay.
Recall that we can define the space of $k$-linear maps
$L^k(\ell^\infty(\mathbb{R}^n), \ell^\infty(\mathbb{R}^n) )$ via the
identification
\[L^k(\ell^\infty(\mathbb{R}^n), \ell^\infty(\mathbb{R}^n))=L(\ell^\infty(\mathbb{R}^n),
L^{k-1}(\ell^\infty(\mathbb{R}^n), \ell^\infty(\mathbb{R}^n) )). \]
There are $k$ possible identifications defined by the isomorphisms $\iota_j$ as
follows. Given the map
\begin{equation}
  \iota_j:L^k(\ell^\infty(\mathbb{R}^n), \ell^\infty(\mathbb{R}^n) )\to
  L(\ell^\infty(\mathbb{R}^n), L^{k-1}(\ell^\infty(\mathbb{R}^n),
  \ell^\infty(\mathbb{R}^n) )), \quad  1\leq j\leq k,\nonumber
\end{equation}
and $A\in
L^k(\ell^\infty(\mathbb{R}^n),\ell^\infty(\mathbb{R}^n))$, let
\begin{equation}
  \iota_j(A)(w)(v_1,\ldots,v_{k-1})=A(v_1,\ldots,\overbrace{w}^j,\ldots,v_{k-1})  \label{LandM:Definition:InclusionIota}.
\end{equation}
The maps $\iota_j,\, 1\leq j \leq k,$ are isometries in the corresponding
operator norms. We define
\begin{align}
  L_\Gamma^k(\ell^\infty(\mathbb{R}^n), \ell^\infty(\mathbb{R}^n))=\big\{&A\in
  L^k(\ell^\infty(\mathbb{R}^n),\ell^\infty(\mathbb{R}^n) )\,\vert\,
  \nonumber\\ &\iota_p(A)\in
  L_\Gamma(\ell^\infty(\mathbb{R}^n),L^{k-1}(\ell^\infty(\mathbb{R}^n),
  \ell^\infty(\mathbb{R}^n) )), \, 1\leq p\leq k\big\},\nonumber
\end{align}
with the norm
\[
\|A\|_\Gamma=\max\{\|A\|,\gamma(A)\},
\]
where 
\[
\gamma(A)=\max_{1\leq p\leq k}\{\gamma(\iota_p(A))\}.
\]

\begin{remark}
Note that this definition is consistent because we can identify
$L^{k-1}(\ell^\infty(\mathbb{R}^n), \ell^\infty(\mathbb{R}^n))$ with
the $\ell^\infty$ space
$\ell^\infty(L^{k-1}(\ell^\infty(\mathbb{R}^n),\mathbb{R}^n))$.  
\end{remark}

With the definition of this norm we can prove that
$L^k_\Gamma(\ell^\infty(\mathbb{R}^{n}),\ell^\infty(\mathbb{R}^{n}))$
is a Banach space using the same tools as in the proof of Proposition
\ref{LandM:Proposition:LGBanachSpace}.

The following proposition gives bounds to the norm of multilinear contractions.
These bounds are fundamental later on, since multilinear contractions appear
naturally when differentiating repeatedly invariance equations, a basic
step in the study of normal form equations and the parameterisation method.

\begin{proposition}\label{LandM:Proposition:ContractionNorm}
  Let $A\in L^k_\Gamma(\ell^\infty(\mathbb{R}^n),
  \ell^\infty(\mathbb{R}^n))$, $k\geq 2$, and $u_1,\ldots,
  u_p\in\ell^\infty(\mathbb{R}^n)$, $1\leq p\leq k-1$. Then, for any
  permutation of $k$ elements $\tau\in S_k$, the map
\[
B_{\tau, u_1, \ldots, u_p}:\ell^\infty(\mathbb{R}^n)\times\stackrel{(k-p)}{\cdots}\times\ell^\infty(\mathbb{R}^n)
\to\ell^\infty(\mathbb{R}^n)
\]
defined by 
\[
B_{\tau, u_1, \ldots, u_p}(v_1,\ldots,v_{k-p})=A(\tau(v_1,\ldots,v_{k-p},u_1,\ldots,
u_p))
\] 
belongs to $L^{k-p}_\Gamma(\ell^\infty(\mathbb{R}^n),
\ell^\infty(\mathbb{R}^n))$. Moreover
\[
\gamma(B_{\tau, u_1, \ldots, u_p})\leq
\gamma(A)\|u_1\|\cdots\|u_p\|
\]
and
\[
\|B_{\tau, u_1, \ldots, u_p}\|_\Gamma \leq
\|A\|_\Gamma\|u_1\|\cdots \|u_p\|.
\]
\end{proposition}

Proposition \ref{LandM:Proposition:ContractionNorm} can be used
to bound $\Gamma$-norms of contractions in such a way that decay
properties can be ignored except for the bound of just one component,
as the next proposition shows.  

{From} Propositions~\ref{LandM:Proposition:LGAlgebraProperty} and 
\ref{LandM:Proposition:ContractionNorm} 
we also obtain
the following composition property, which will prove crucial for later
developments.

Given $A\in
L^k_{\Gamma}(\ell^{\infty}(\mathbb{R}^n),\ell^{\infty}(\mathbb{R}^n))$,
$B_{j} \in
L^{l_j}_{\Gamma}(\ell^{\infty}(\mathbb{R}^n),\ell^{\infty}(\mathbb{R}^n))$
for $j=1, \dots, k$ and $w_{l_j}\in\ell^\infty(\mathbb{R}^n)^{l_j}$,
we define the composition $A B_{1} \cdots B_{k}$ by \[A B_{1} \cdots
B_{k}(w_{l_1},\ldots,w_{l_k})=A(B_1w_{l_1},\ldots, B_kw_{l_k}).\]

\begin{proposition}
  \label{LandM:Proposition:LGAlgebraPropertyMultilinear} If $A\in
  L^k_{\Gamma}(\ell^{\infty}(\mathbb{R}^n),\ell^{\infty}(\mathbb{R}^n))$
  and $B_{j} \in
  L^{l_j}_{\Gamma}(\ell^{\infty}(\mathbb{R}^n),\ell^{\infty}(\mathbb{R}^n))$,
  for $j=1, \dots, k$, then the composition $A B_{1} \cdots B_{k} \in
  L^{l_1+\cdots
    +l_k}_{\Gamma}(\ell^{\infty}(\mathbb{R}^n),\ell^{\infty}(\mathbb{R}^n))$
  and
  \begin{align}
    \label{LandM:Proposition:LGAlgebraPropertyMultilinear:normacompklineals-gamma} \gamma (A B_{1} \cdots B_{k})  \leq
    & \gamma (A) \| B_{1}\|_{\Gamma}\cdots \| B_{k}\|_{\Gamma}, \\
    \label{LandM:Proposition:LGAlgebraPropertyMultilinear:normacompklineals} \| A B_{1} \cdots B_{k}\|_{\Gamma} \leq &
    \| A\|_{\Gamma} \| B_{1}\|_{\Gamma}\cdots \| B_{k}\|_{\Gamma}.
  \end{align}
  A consequence of the proof is that if $B_p=B_q$ for all $q\neq p$ the bounds can be written instead as 
\begin{align}
    \gamma (A B \cdots B)  \leq
    & \gamma (A) \| B\|_{\Gamma}\|B\|\cdots \|B\|, \nonumber\\
    \| A B\cdots B\|_{\Gamma} \leq &
    \| A\|_{\Gamma} \| B\|_{\Gamma}\|B\|\cdots \|B\|.\nonumber
  \end{align}
\end{proposition}
An important special case of the previous proposition is the following
result, which follows from the proofs of Propositions \ref{LandM:Proposition:LGAlgebraProperty} and \ref{LandM:Proposition:LGAlgebraPropertyMultilinear}.

\begin{corollary}
  If $A\in L_\Gamma(\ell^\infty(\mathbb{R}^n), \ell^\infty(\mathbb{R}^n))$ and $B\in L^k_\Gamma(\ell^\infty(\mathbb{R}^n), \ell^\infty(\mathbb{R}^n))$ then $A\cdot B\in L^k_\Gamma(\ell^\infty(\mathbb{R}^n), \ell^\infty(\mathbb{R}^n))$ and
  \begin{align}
    \gamma(A\, B)\leq \gamma(A)\gamma(B),\nonumber\\
    \|A\, B\|_\Gamma\leq \|A\|_\Gamma\|B\|_\Gamma.\nonumber
\end{align}
\end{corollary}



\section{Spaces of differentiable functions with decay}
\label{Section:DiffFunDecay}

With the definitions of linear and multilinear applications with decay from the
previous section we are prepared to define spaces of differentiable functions
with decay whose domain is an open set in $\ell^\infty(\mathbb{R}^n)$.

\begin{definition}
  Let $U$ be an open set of $\ell^\infty(\mathbb{R}^n)$. We define
  \begin{align}
    C^1_\Gamma(U,\ell^\infty(\mathbb{R}^n))=\{&F\in
    C^1(U,\ell^\infty(\mathbb{R}^n))\,\vert\, \sup_{x\in
      U}\|F(x)\|_\infty<\infty,\nonumber\\ &DF(x)\in
    L_\Gamma(\ell^\infty(\mathbb{R}^{n}),\ell^\infty(\mathbb{R}^{n})),\forall
    x\in U,\nonumber\\ & \sup_{x\in
      U}\|DF(x)\|_\Gamma<\infty\}\nonumber
  \end{align}
  with norm 
  \begin{align}
    \|F\|_{C^1_\Gamma }=\max\big(\|F\|_{C^0}, \sup_{x\in U}\|DF(x)\|_\Gamma\big),\nonumber
  \end{align}
  where $\|F\|_{C^0}=\sup_{x\in U}\|F(x)\|_{\infty}$ as usual. We can also define 
  \begin{align}
    C^1_\Gamma(U,L^k(\ell^\infty(\mathbb{R}^n),
    \ell^\infty(\mathbb{R}^n)))=\big\{&F\in
      C^1(U,L^k(\ell^\infty(\mathbb{R}^n),\ell^\infty(\mathbb{R}^n)))\,\vert\,\nonumber\\
      &F(x)\in
      L^k_\Gamma(\ell^\infty(\mathbb{R}^n),\ell^\infty(\mathbb{R}^n)),\,
      \forall x\in U\nonumber,\\ &\sup_{x\in U}\|F(x)\|_\Gamma<\infty\big\}.\nonumber
  \end{align}
  This definition is consistent since
$$
L^k(\ell^\infty(\mathbb{R}^n),\ell^\infty(\mathbb{R}^n)) \sim 
\ell^\infty(L^k(\ell^\infty(\mathbb{R}^n),\mathbb{R}^n)).
$$
  
  Based on the above, we define spaces of $C^r_\Gamma$ functions as:
  \begin{align*}
    C^r_\Gamma(U,\ell^\infty(\mathbb{R}^n))=\big\{F\in
    C^r(U,\ell^\infty(\mathbb{R}^n))\,\vert\, D^kF\in 
    C^1_\Gamma(U, & L^k(\ell^\infty(\mathbb{R}^{n}),\ell^\infty(\mathbb{R}^{n}))), \\
    & \,0\leq k \leq r-1\big\}
  \end{align*}
  with norm
  \begin{align}
    \|F\|_{C^r_\Gamma}=\max \left(\|F\|_{C^0},\max_{0\leq k\leq r-1}\sup_{x\in U}\|D D^kF(x)\|_\Gamma\right).\nonumber
  \end{align}
\end{definition}
\begin{remark}
  The inclusions $C^i_\Gamma \subset C^{i-1}_\Gamma,\, 1\leq i\leq r,$
are satisfied.
\end{remark}

It is easy to check that $C^r_\Gamma(U,\ell^\infty(\mathbb{R}^n))$ is
a Banach space. We have the following result concerning the
composition of maps, which can be found in
\cite{FontichDelaLlaveMartinFAFramework}.

\begin{proposition}
  Let $U, V$ be open sets of $\ell^\infty(\mathbb{R}^n)$, $F\in
  C^r_\Gamma(U,\ell^\infty(\mathbb{R}^n))$ and $G\in C^r_\Gamma(V,
  \ell^\infty(\mathbb{R}^n))$ such that $F(U)\subseteq V$. Then
\begin{itemize}
\item[(1)] $G\circ F \in C^r_\Gamma(U, \ell^\infty(\mathbb{R}^n) )$,
\item[(2)] $\|G\circ F\|_{C^r_\Gamma}\leq
  K(1+\|F\|^r_{C^r_\Gamma})\|G\|_{C^r_\Gamma}$.
\end{itemize}
\end{proposition}

\begin{remark}\label{DiffFunDecay:Remark:CompositionLinearMap}
  An important particular case appears when $G$ is a linear map in
  $L_\Gamma(\ell^\infty(\mathbb{R}^n), \ell^\infty(\mathbb{R}^n))$. In
  this case the estimates in the proof are much easier and the bound
  is
  \begin{equation}
    \|A \circ F\|_{C^r_\Gamma}\leq \|A\|_\Gamma\|F\|_{C^r_\Gamma}.
  \end{equation}
\end{remark}

\begin{theorem}[Inverse Function Theorem]\label{DiffFunDecay:InverseFunctionTheorem}
  Let $U$ be an open set of $\ell^\infty(\mathbb{R}^n)$ and $F\in
  C^r_\Gamma(U, \ell^\infty(\mathbb{R}^n))$, $r\geq 1$. Let $p\in U$
  and $q=F(p)$. Assume that $DF(p)$ is invertible and $DF(p)^{-1}\in
  L_\Gamma(\ell^\infty(\mathbb{R}^n),
  \ell^\infty(\mathbb{R}^n))$. Then $F$ is locally invertible around
  $p$ and $F^{-1}\in C^r_\Gamma(V, \ell^\infty(\mathbb{R}^n))$, where
  $V$ is a suitable neighbourhood of $q$.
\end{theorem}

  We defer the proof of this result until Section \ref{Preliminaries:Functional}.



\section{Spectral theory for $\Gamma$-coupled linear maps}
\label{Preliminaries:Functional}
In this section we recall some results in spectral theory of linear
operators and also introduce a new notion, the $\Gamma$-spectrum of a
linear operator in a lattice, associated to a decay function $\Gamma$
satisfying the definition introduced in Section
\ref{Section:DecayFuns}.

Given a Banach space $E$ the space of continuous linear maps $L(E, E)$
is also a Banach space with the standard operator norm. Moreover, it is
a Banach algebra with the product given by the composition of maps. 

Let $E=\ell^\infty(\mathbb{R}^n)$. Given a decay function $\Gamma$ as
in Definition \ref{DecayFuns:Definition:DecayFunction}, the space
$L_\Gamma(E,E)$ introduced in Section
\ref{Section:LinearAndMultilinearWithDecay} is a Banach algebra (see
Proposition \ref{LandM:Proposition:LGAlgebraProperty}).

The inclusion $L_\Gamma(E,E)\subset L(E,E)$ holds considering both spaces as
sets, but $L_\Gamma(E,E)$ is not a closed subalgebra of $L(E,E)$, hence it is
not a Banach subalgebra of $L(E,E)$. Indeed, consider a specific decay function:
\[\Gamma(j)=a|j|^{-\alpha}e^{-\theta|j|},\qquad j\in\mathbb{Z}^m,\]
with $\alpha>m,\,\theta>0$ and $a>0$ small enough.

Consider the sequence of linear maps $\{A^k\}_{k\in\mathbb{N}}$ defined by
\begin{equation*}
  A^k=\begin{cases}
    A^k_{i,j}=|i-j|\Gamma(i-j),\qquad |i-j|\leq k,\\
    A^k_{i,j}=0,\qquad \mathrm{otherwise.}
  \end{cases}
\end{equation*}

Clearly $A^k\in L_\Gamma(E,E)$. Next we check that $\{A^k\}_{k\in\mathbb{N}}$
converges to $A^\infty$ in $L(E,E)$, where
$A^\infty_{i,j}=|i-j|\Gamma(i-j),\,\forall i,j\in\mathbb{Z}^m$. Indeed
\begin{align}
\|A^\infty-A^k\|&=\sup_{\substack{u\in E\\ \|u\|\leq 1}}\|(A^\infty-A^k)u\|=\sup_{\|u\|\leq1}\sup_{i\in\mathbb{Z}^m}\|\sum_{|i-j|>k}|i-j|\Gamma(i-j)u_j\|\nonumber\\
&\leq \sum_{|l|>k}|l|\Gamma(l)\nonumber
\end{align}
which goes to zero as $k\to\infty$ because
$\sum_{l\in\mathbb{Z}^m}|l|\Gamma(l)$ is convergent provided either
$\theta>0$ or $\theta=0$ and $\alpha>m+1$. However $A^\infty\notin L_\Gamma(E,E)$ because
\[
\gamma(A^\infty)=\sup_{i,j\in\mathbb{Z}^m}|A_{i,j}|\Gamma(i-j)^{-1}=\sup_{i,j\in\mathbb{Z}^m}|i-j|=\infty.
\]

The space $L_\Gamma(E,E)$ is a Banach algebra with the identity as
unit, but $\|\Id\|_\Gamma=\Gamma(0)^{-1}\neq 1$. To be able to apply the general
results of Banach algebras with unit, we can introduce an equivalent norm in
$L_\Gamma(E,E)$, say $\|\cdot\|'$, such that $\|\Id\|'=1$. The procedure is
standard (see \cite{BEJohnsonBanach}). We define
\[\|A\|'=\sup\left\{\|AC\|_\Gamma,\, C\in L_\Gamma(E,E),\, \|C\|_\Gamma\leq 1\right\}.\] The
properties of norm are easily checked from the definition, proving the
equivalence requires the following. On one hand,

\[\|A\|'=\sup_{\|C\|_\Gamma\leq 1}\|AC\|_\Gamma\leq \sup_{\|C\|_\Gamma\leq
  1}\|A\|_\Gamma\|C\|_\Gamma=\|A\|_\Gamma,\] on the other hand,
\[\|A\|'\geq \|A\frac{\Id}{\|\Id\|_\Gamma}\|_\Gamma=\frac{1}{\|\Id\|_\Gamma}\|A\|_\Gamma.\]
Finally,
\[\|\Id\|'=\sup_{\|C\|_\Gamma\leq 1}\|\Id\cdot C\|_\Gamma=\sup_{\|C\|_\Gamma\leq
  1}\|C\|_\Gamma=1.\]

To illustrate some features of $L_\Gamma(E, E)$ we
present an example of an invertible linear map in
$L_\Gamma(\ell^\infty(\mathbb{C}^n), \ell^\infty(\mathbb{C}^n))$ such
that its inverse may not be in $L_\Gamma(\ell^\infty(\mathbb{C}^n),
\ell^\infty(\mathbb{C}^n))$ depending on the decay function $\Gamma$
considered.

Let $\ell^\infty(\mathbb{C})$ be a one dimensional lattice ($m=1$),
$r\in\mathbb{N}$, $a_0, \ldots, a_r\in \mathbb{C}$ and $A\in
L(\ell^\infty(\mathbb{C}^n), \ell^\infty(\mathbb{C}^n))$ determined by

\begin{align*}
A_{i,j}&=0, \qquad \text{if either }j<i\text{ or }j>i+r,\\
A_{i,i+k}&=a_k,\qquad 0\le k\le r,
\end{align*}
with $(Ax)_i=\sum_{j=i}^{i+r}A_{ij}x_j$.

Clearly $A\in L_\Gamma(\ell^\infty(\mathbb{C}),
\ell^\infty(\mathbb{C}))$ for any decay function $\Gamma$, since
\begin{align}
   \gamma(A)=\sup_{i,j}|A_{ij}|\Gamma(i-j)^{-1}=\max\{|a_0|\Gamma(0)^{-1}, |a_1|\Gamma(1)^{-1}, \ldots, |a_r|\Gamma(r)^{-1}\}<\infty\nonumber
\end{align}
and 
\begin{align}
  \|A\|=\sup_{\|x\|\leq 1}\|Ax\|=\sup_{\|x\|\leq
    1}\sup_{i\in\mathbb{Z}}\|(\ldots,
  \sum_{j=i}^{i+r}A_{ij}x_j,\ldots)\|=|a_0|+\ldots+|a_r|.\nonumber
\end{align}
We look for the inverse $B$ of $A$ assuming {\it a priori} that the inverse is upper
triangular and a band matrix. That is, $B_{ij}=b_{j-i}$ for some
$b_k\in\mathbb{C}$, with $b_k=0$ if $k<0$.

Imposing the condition $AB=\id$, or equivalently
\[\sum_{k\in\mathbb{Z}}A_{ik}B_{kj}=\delta_{ij}\]
we get 
\[a_0b_{j-i}+a_1b_{j-i-1}+\ldots+a_rb_{j-i-r}=\delta_{ij}.\] When
$i=j$ we have $a_0b_0=1$. This condition implies $a_0\neq 0$. We
assume it from now on. Then we proceed by induction and recursively
obtain $b_j$ for $j>0$. Actually $b_j$ satisfies the $r$-th order
linear difference equation
\[b_j=-\frac{a_1}{a_0}b_{j-1}-\frac{a_2}{a_0}b_{j-2}-\ldots-\frac{a_r}{a_0}b_{j-r},\qquad j\geq 1,\]
with initial conditions $b_0=1/a_0$, $b_{-1}=0,\ldots,b_{-r+1}=0$.

Using the theory of linear difference equations we can compute $b_j$
in terms of the zeros of the characteristic polynomial of this equation,
\[a_0x^r+a_1x^{r-1}+\ldots+a_r=0.\] Once we have determined $b_j$ and hence $B$, we
can check that indeed \[AB=BA=\id.\] For this to hold we strongly use
that $A_{ik}\in L(\mathbb{R},\mathbb{R})\sim \mathbb{R}$. It
remains to check that $B$ sends $\ell^\infty(\mathbb{R})$ to
itself. This will depend on the choice of the values of $a_i$.

To work with a specific example, assume $r=2$ and $a_0=1$. Hence we
can determine the zeros of the characteristic polynomial and write the
general solution of the difference equation as

\[b_j=\beta_1\left(\frac{-a_1+\sqrt{a_1^2-4a_2}}{2}\right)^j+\beta_2\left(\frac{-a_1-\sqrt{a_1^2-4a_2}}{2}\right)^j,\qquad
j\geq 0,\] for suitable values $\beta_1, \beta_2$.

Now we can choose $a_1, a_2$ to adjust the growth of the coefficients
$b_j$. For instance, taking $a_1=-\frac{3}{4}$, $a_2=\frac{1}{8}$,
then $b_j=2\left(\frac{1}{2}\right)^j-\left(\frac{1}{4}\right)^j$. In
this case $B\in L(\ell^\infty(\mathbb{R}), \ell^\infty(\mathbb{R}))$,
because $\sum_{j\geq0}|b_j|<\infty$. With the choice of decay function
$\Gamma(j)=a|j|^{-\alpha}e^{-\theta|j|}$ we have that
\begin{align}
  \gamma(B)&=\sup_{i,j}|B_{ij}|\Gamma(i-j)^{-1}\nonumber\\
  &=\max\left(\frac{1}{a},\sup_{j-i\geq 1}
    \left[2\left(\frac{1}{2}\right)^{j-i}-\left(\frac{1}{4}\right)^{j-i}\right]
  a^{-1}|j-i|^\alpha e^{\theta|j-i|}\right)\nonumber
\end{align}
which is finite when $\theta<\log 2$. Hence $B\in
L_\Gamma(\ell^\infty(\mathbb{R}), \ell^\infty(\mathbb{R}))$ (with this
particular choice of decay function $\Gamma$) if and only if
$\theta<\log 2$.

\subsection{$\Gamma$-spectrum of linear maps on lattices}
\label{sec:Gamma-spec}
Consider the lattice $\ell^\infty(\mathbb{R}^n)$ and a decay function
$\Gamma$. Also consider the complexified space
\[\ell^\infty(\mathbb{R}^n)\otimes_\mathbb{R}\mathbb{C}\sim\ell^\infty(\mathbb{R}^n)\oplus
i\ell^\infty(\mathbb{R}^n)\sim\ell^\infty(\mathbb{C}^n).\]

Let $\mathcal{E}$ be a linear subspace of
$\ell^\infty(\mathbb{C}^n)$. Given $A\in
L_\Gamma(\mathcal{E},\mathcal{E})$ we define:

\begin{itemize}
\item $\Gamma$-resolvent  of $A$ as
  \[\rhogamma(A)=\{\lambda\in\mathbb{C}\,\vert\,
  A-\lambda\id\text{ is invertible and }(A-\lambda\Id)^{-1}\in
  L_\Gamma(\mathcal{E},\mathcal{E})\},\]
\item $\Gamma$-spectrum of $A$ as
  \[\spec_\Gamma (A)=\mathbb{C}\backslash\rhogamma(A),\]
\item $\Gamma$-spectral radius of $A$ as
  \[\rgamma(A)=\sup\{|\lambda|\,\vert\, \lambda\in\spec_\Gamma(A)\}.\]
\end{itemize}

From the definitions above it is immediate that 
\[\rhogamma(A)\subset \rho(A)\]
and therefore 
\[\spec(A)\subset\spec_\Gamma(A),\qquad r(A)\leq \rgamma(A).\]

Also from the definitions, it is clear that $\spec_\Gamma(A)$ is the spectrum of
$A$ as an element of the Banach algebra $L_\Gamma(\mathcal{E}, \mathcal{E})$.
Hence all general properties of spectra in unitary Banach algebras apply to
$\spec_\Gamma$.

\subsection{Operational calculus}

The results in this section are similar to those for the spectrum of a
linear operator, a standard reference is \cite{SchechterFunctionalAnalysis}.
For the remaining of this section let $\mathcal{E}$ denote a complex Banach space.

Let $A\in L_\Gamma(\mathcal{E}, \mathcal{E})$ and $\Omega$ be an open set such that
$\spec_\Gamma(A)\subset\Omega$. Let $\omega$ be an open set such
that
\begin{equation}
  \spec_\Gamma(A)\subset \omega\subset \overline{\omega}\subset \Omega\label{Preliminaries:Operational:Eq:SpecInclusion}
\end{equation}
and $\partial \omega$ is a finite union of closed curves.

Then, given $f:\Omega\to\mathbb{C}$ analytic we define
\[f(A)=\frac{1}{2\pi i}\int_{\partial\omega}f(z)(z-A)^{-1}\dd z.\]

This definition is independent of the choice of $\omega$ provided it
satisfies the conditions above, and we have that

\[f(A)\in L_\Gamma(\mathcal{E},\mathcal{E}).\]

In the case that $f$ is a polynomial, $f(z)=\sum_{k=0}^ma_kz^k$, the
previous definition gives $f(A)=\sum_{k=0}^ma_kA^k$.

The following proposition proves upper semicontinuity of $\spec_\Gamma (A)$
with respect to $A$.

\begin{proposition}\label{Preliminaries:Operational:Prop:Continuity}
Let $A\in L_\Gamma(\mathcal{E}, \mathcal{E})$ and $\mu\in
\rhogamma(A)$. Then if $B\in L_\Gamma(\mathcal{E}, \mathcal{E})$
   and $\|B\|_\Gamma$ is small enough, then $\mu\in\rhogamma(A+B)$.
 \end{proposition}

Moreover, if $f, g:\Omega\to \mathbb{C}$ are analytic functions and
$h(z)=f(z)g(z)$, hence
\begin{equation}
   h(A)=f(A)g(A).\label{Preliminaries:Operational:EqProd}
\end{equation}
As a consequence, under these conditions if we assume $f:\Omega\to
\mathbb{C}$ is not zero, thus $f(A)$ is invertible, $f(A)^{-1}\in
L_\Gamma(\mathcal{E}, \mathcal{E})$ and
\begin{equation}
  \left[f(A)\right]^{-1}=\frac{1}{2\pi i}\int_{\partial\omega}\frac{1}{f(z)}(z-A)^{-1}\dd z\label{Preliminaries:Operational:EqInv}
\end{equation}
where $\omega$ satisfies \eqref{Preliminaries:Operational:Eq:SpecInclusion}.

With these results on spectra of $\Gamma$-linear maps we can now prove
the inverse function theorem in lattices with spatial decay.

\begin{proof}[Proof of Theorem \ref{DiffFunDecay:InverseFunctionTheorem}]

  From the standard inverse function theorem in Banach spaces, $F$ is
  locally invertible and $F^{-1}$ is defined in a neighbourhood $V$ of
  $q$. Moreover, $DF^{-1}(q)=DF(p)^{-1}$ and by the continuity of $DF$
  and continuity of $\spec_\Gamma$, $DF^{-1}(x)\in L_\Gamma$ for $x\in
  V$, provided $V$ is small. Since
  \begin{equation}
    DF^{-1}(x)=\left(DF(F^{-1}(x))\right)^{-1}\label{DiffFunDecay:InverseFunctionTheorem:Eq1}
  \end{equation}
  we can obtain the higher order derivatives of $F^{-1}$ by taking
  derivatives in the right hand side of
  \eqref{DiffFunDecay:InverseFunctionTheorem:Eq1}. For instance,
  \begin{align}
    D^2F^{-1}(x)&=-\left(DF(F^{-1}(x))\right)^{-1}D^2F(F^{-1}(x))(DF(F^{-1}(x)))^{-1}\nonumber\\
    &=-DF^{-1}(x)D^2F(F^{-1}(x))DF^{-1}(x).\label{DiffFunDecay:InverseFunctionTheorem:Eq2}
  \end{align}

  Then, by Proposition
  \ref{LandM:Proposition:LGAlgebraPropertyMultilinear}, we have
  $D^2F^{-1}(x)\in L^2_\Gamma$. Proceeding in the same way for the
  other derivatives we get that $F^{-1}\in C^r_\Gamma(V,
  \ell^\infty(\mathbb{R}^n))$. Alternatively, we can use
  \eqref{DiffFunDecay:InverseFunctionTheorem:Eq2} to prove inductively
  that $F^{-1}\in C^i_\Gamma$ assuming $F^{-1}\in C^{i-1}_\Gamma$, for
  $i\leq r$.
  
\end{proof}

\subsection{Spectral projections associated to a gap in the $\Gamma$-spectrum}

We can adapt the spectral projection theorem (see for instance
\cite{SchechterFunctionalAnalysis}) to the setting of $\Gamma$-spectrum. The
statements and proofs are very similar to the ones corresponding to
$L(\mathcal{E},\mathcal{E})$ but we give them here for the sake of completeness.

Assume that \[\spec_\Gamma(A)=\sigma_1\cup\sigma_2,\]
with \[\sigma_i\subset
\omega_i\subset\overline{\omega_i}\subset\Omega_i,\qquad i=1,2,\] where
$\Omega_i$ are disjoint open sets and $\omega_i$ are open sets such
that $\partial\omega_i$ are finite union of simple closed curves.

We define
\[P=\frac{1}{2\pi i}\int_{\partial\omega_1}(z-A)^{-1}\dd z.\] Results on
integration of elements of $L_\Gamma$ can be found in
\cite{FontichDelaLlaveMartinFAFramework}.

\begin{lemma}\label{GapProjections:Lemma:Projector}
We have
\begin{enumerate}[label=(\roman*)]
\item $P\in L_\Gamma(\mathcal{E}, \mathcal{E})$,\label{Preliminaries:Operational:Lemma:Projection:IsGamma}
\item $P^2=P$\label{Preliminaries:Operational:Lemma:Projection:IsProjection},
\item $P(\mathcal{E})$ and $\krn(P)$ are closed and invariant.\label{Preliminaries:Operational:Lemma:Projection:Invariants}
\end{enumerate}
\end{lemma}

\begin{proof}
  Part \ref{Preliminaries:Operational:Lemma:Projection:IsGamma}
  follows from the properties of integrals of functions in the Banach
  algebra $L_\Gamma$ in \cite{FontichDelaLlaveMartinFAFramework}. Part
  \ref{Preliminaries:Operational:Lemma:Projection:IsProjection}
  follows from the fact that $P$ can be written as
  \[\frac{1}{2\pi i }\int_{\partial\omega}f(z)(z-A)^{-1}\dd z,\]
  with $f:\Omega_1\to \mathbb{C}$, defined by $f(z)=1$. Since
  $f(z)=f(z)f(z)$, by \eqref{Preliminaries:Operational:EqProd}
  \[PP=f(A)f(A)=f(A)=P,\]
  proving $P$ is a projection.
  
  For Part
  \ref{Preliminaries:Operational:Lemma:Projection:Invariants},
  $P(\mathcal{E})$ and $\krn(P)$ are invariant when $P$ is a projection,
  and $\mathcal{E}=P(\mathcal{E})\oplus \krn(P)$.  Moreover since
  $P(\mathcal{E})=\krn (\id -P)$ and $P$ is continuous, both
  $P(\mathcal{E})$ and $\krn(\id -P)$ are closed.
  
\end{proof}

We denote $\mathcal{E}^1=P(\mathcal{E})$ and
$\mathcal{E}^2=(\id-P)(\mathcal{E})=\krn (P)$ and
$A_i=\left.A\right|_{\mathcal{E}_i}$.

\begin{theorem}\label{GapProjections:Theorem:Spectrum}
We have that
\[\spec_\Gamma(A_i)=\sigma_i,\qquad i=1,2.\]
\end{theorem}

\begin{proof}
Let
\[f(z)=\begin{cases}1,\quad \text{ if }z\in\Omega_1,\\0, \quad \text{ if
    }z\in\Omega_2.\end{cases}\] Moreover, let $\lambda\notin\sigma_1$ and
$g_1(z)=\frac{f(z)}{\lambda-z}$. The function $g_1$ is analytic in a
neighbourhood $U$ of $\spec_\Gamma (A)$ and satisfies
\begin{align}
f(z)&=(\lambda-z)g_1(z),\nonumber\\
f(z)g_1(z)&=g_1(z)\nonumber,
\end{align}
for $z\in U$.

By \eqref{Preliminaries:Operational:EqProd}, 
\begin{align}
f(A)&=(\lambda-A)g_1(A),\label{Preliminaries:Operational:Theorem:Spec:Eq}\\
f(A)g_1(A)&=g_1(A),\nonumber
\end{align}
and hence
\begin{align}
  P&=(\lambda-A)g_1(A)=g_1(A)(\lambda-A),\\ Pg_1(A)&=g_1(A)P=g_1(A).
\end{align}
If $x\in \mathcal{E}^1$, \[g_1(A)x=Pg_1(A)x=P(g_1(A)x)\in \mathcal{E}^1.\] Moreover, from \eqref{Preliminaries:Operational:Theorem:Spec:Eq}
\[g_1(A)=\left.(\lambda-A)^{-1}\right|_{\mathcal{E}^1}=(\lambda-A_1)^{-1}\]
which implies that $\lambda\in \rhogamma(A_1)$. Therefore
$\spec_\Gamma(A_1)\subset \sigma_1$. Since $\id-P$ is a projection onto
$\mathcal{E}^2$ a completely analogous argument shows that
$\spec_\Gamma(A_2)\subset\sigma_2.$ Indeed, given $\lambda \notin\sigma_2$, let
$g_2(z)=\frac{1-f(z)}{\lambda-z}.$ Then
\begin{align*}
  (\id-P)&=(\lambda-A)g_2(A)=g_2(A)(\lambda-A),\\ (\id-P)g_2(A)&=g_2(A)(\id-P)=g_2(A)
\end{align*}
and
\[g_2(A)=\left.(\lambda-A)^{-1}\right|_{\mathcal{E}^2}=(\lambda-A_2)^{-1}.\]
Now suppose that $\lambda\in\rhogamma(A_1)\cap\rhogamma(A_2).$ Since $g_1(z)+g_2(z)=\frac{1}{\lambda-z}$, by \eqref{Preliminaries:Operational:EqInv} we have
\begin{align}
(\lambda-A)^{-1}&=g_1(A)+g_2(A)=g_1(A)P+g_2(A)(\id-P)\nonumber\\
&=(\lambda-A_1)^{-1}P+(\lambda-A_2)^{-1}(\id-P).\nonumber
\end{align}
Then $\lambda\in\rhogamma(A)$, which implies that
\[\spec_\Gamma(A)\subset \spec_\Gamma(A_1)\cup\spec_{\Gamma}(A_2),\]
and therefore
\[\sigma_i\subset\spec_\Gamma(A_i),\qquad i=1,2.\]
\end{proof}

\section{Sylvester operators in $L^k_\Gamma$}
\label{Preliminaries:SylvesterInLk}
In this section we will introduce Sylvester operators and prove some results  in
spaces of $k$ linear maps with decay. 

\begin{definition}
  Let $E=\ell^\infty(\mathbb{R}^n)$. Given $A, B\in L_\Gamma(E, E)$ we
  define the operators \[\mathcal{R}_{j,A}:L^k_\Gamma(E, E)\to
  L^k_\Gamma(E, E),\,\qquad 1\leq j\leq k,\] by
\[\mathcal{R}_{j,A}(W)(u_1, \ldots, u_k)=W(u_1,\ldots,Au_j,\ldots, u_k),\]
and $\mathcal{L}_B, \mathcal{S}_{B,A}:L^k_\Gamma(E, E)\to
L^k_\Gamma(E, E)$ by
\begin{align}
   \mathcal{L}_{B}(W)(u_1, \ldots, u_k)&=BW(u_1, \ldots, u_k),\nonumber\\
   \mathcal{S}_{B, A}(W)(u_1, \ldots, u_k)&=BW(Au_1, \ldots, Au_k),\nonumber
\end{align}
respectively.
\end{definition}

Note that by Proposition
\ref{LandM:Proposition:LGAlgebraPropertyMultilinear}, if $W\in
L_\Gamma^k(E, E)$ then $\mathcal{R}_{j, A}(W), \mathcal{L}_B(W)$ and
$\mathcal{S}_{B, A}(W)$ are in $L^k_\Gamma(E, E)$ so that the
operators are well defined.

Given two subsets $X$, $Y$ of $\mathbb{C}$ we denote by $X\cdot Y$ the
set
\[X\cdot Y=\{x\cdot y\,\vert\, x\in X,\, y\in Y\}.\] Analogously, we
define
\[X^k=X\cdot\stackrel{k}{\ldots}\cdot X.\]

\begin{proposition}\label{Preliminaries:SylvesterGamma:Proposition}
  We have
   \[\spec(\mathcal{S}_{B, A}, L_\Gamma^k(E, E))\subset \spec_\Gamma(B)\cdot\left(\spec_\Gamma(A)\right)^k, \qquad k\in\mathbb{N}.\]
\end{proposition}

The proof of this proposition is a consequence of the following theorem and the next lemma.

\begin{theorem}\label{ThmRudin}[Theorem 11.23, \cite{RudinFA}]
  Let $\mathfrak{a}$ and $\mathfrak{b}$ be two commuting elements in a unitary Banach
  algebra. Then
  \[\spec(\mathfrak{a}\mathfrak{b})\subseteq\spec(\mathfrak{a})\cdot\spec(\mathfrak{b}).\]
\end{theorem}

\begin{lemma}
  Given $A, B\in L_\Gamma(E, E), k\in\mathbb{N}, 1\leq j\leq k$, then
  \begin{align}
    \spec(\mathcal{R}_{j,A}, L_\Gamma^k(E, E))&\subset\spec_\Gamma(A),\nonumber\\
    \spec(\mathcal{L}_B, L^k_\Gamma(E, E))&\subset\spec_\Gamma(B).\nonumber
  \end{align}
\end{lemma}
\begin{proof}
  Let $\lambda\in \rhogamma(A)$, thus
  $(A-\lambda\id)^{-1}\in L_\Gamma(E, E)$. To study the invertibility
  of $\mathcal{R}_{j, A}-\lambda\id$ we consider the equation
  \[W(u_1, \ldots, A u_j, \ldots, u_k)-\lambda W(u_1, \ldots, u_j,
  \ldots, u_k)=H(u_1, \ldots,u_j, \ldots, u_k),\] for $W,\,H\in
  L^k_\Gamma(E, E)$, which is equivalent to
  \[W(u_1, \ldots, (A-\lambda\id)u_j,\ldots, u_k)=H(u_1, \ldots, u_j,
  \ldots, u_k).\] Formally,
  \[W=\mathcal{R}_{j, (A-\lambda\id)^{-1}}H\] and hence $W\in
  L^k_\Gamma(E, E)$ and $\lambda\in \rho(\mathcal{R}_{j, A})$.

  The proof of the result for $\mathcal{L}_B$ is completely analogous.

\end{proof}

\begin{proof}[Proof of Proposition
  \ref{Preliminaries:SylvesterGamma:Proposition}]
  It follows directly from the fact that
  \begin{equation}
    \label{Preliminaries:Functional:Equation:Sylvester}
\mathcal{S}_{B, A}=\mathcal{L}_B\circ\mathcal{R}_{1,A}\circ\ldots\circ \mathcal{R}_{k, A}
  \end{equation}
  and the fact that the operators on the r.h.s. of
  \eqref{Preliminaries:Functional:Equation:Sylvester} commute. Then Theorem
  \ref{ThmRudin} proves the result.
  
\end{proof}



\section{Normal forms of maps in lattices}
\label{Section:NormalFormsSternberg}

In this section we consider the computation of normal forms around a
fixed point of a map in a lattice, assuming the map has decay
properties. We estimate the decay properties of the normal form and
the transformation leading to it.

To study the decay properties of normal forms we will use Sylvester operators in
spaces with decay, introduced in the previous section.

We consider an open set $U$ of $\ell^\infty(\mathbb{R}^n)$ such that
$0\in U$ and a map 
\[F:U\to\ell^\infty(\mathbb{R}^n)\]such that $F(0)=0$ and $F\in
C^r_\Gamma(U, \ell^\infty(\mathbb{R}^n))$. Let $A=DF(0)$, with $A\in
L_\Gamma(\ell^\infty(\mathbb{R}^n), \ell^\infty(\mathbb{R}^n))$, invertible and
consider its $\Gamma$-spectrum $\spec_\Gamma(A)$.

\begin{theorem}\label{Sternberg:NormalForms:Theorem}
  In the previously described setting there exist polynomials $K\in
  C^\infty_\Gamma(\ell^\infty(\mathbb{R}^n),
  \ell^\infty(\mathbb{R}^n))$ and $H\in
  C^\infty_\Gamma(\ell^\infty(\mathbb{R}^n),
  \ell^\infty(\mathbb{R}^n))$ of degree at most $r$ such that
  $K(0)=0$, $DK(0)=\id$ and
  \[F\circ K(x)-K\circ H(x)=o(\|x\|^r)\] and $H(x)=Ax+\sum_{j\in
    J}H_jx^{\otimes j}$ with $H_j\in
  L_\Gamma^j(\ell^\infty(\mathbb{R}^n), \ell^\infty(\mathbb{R}^n))$ where
  \[J=\{2\leq j\leq r\,\vert\,
  (\spec_\Gamma(A))^j\cap\spec_\Gamma(A)\neq \emptyset\}.\]
\end{theorem}

\begin{corollary}\label{Sternberg:NormalForms:Corollary}
  Under the conditions of the previous theorem,
  if \[(\spec_\Gamma(A))^j\cap\spec_\Gamma(A)=\emptyset,\qquad 2\leq
  j\leq r,\] then there exists a polynomial $K\in
  C^\infty_\Gamma(\ell^\infty(\mathbb{R}^n),
  \ell^\infty(\mathbb{R}^n))$, $K(0)=0$, $DK(0)=\id$, such that \[F\circ K(x)-K\circ
  Ax=o(\|x\|^r).\]
\end{corollary}

\begin{proof}[Proof of Theorem \ref{Sternberg:NormalForms:Theorem}]
  We look for $K$ and $H$ in the form 
  \begin{align}
     K(x)=\sum_{j=1}^r K_j x^{\otimes j},\nonumber\\
     H(x)=\sum_{j=1}^r H_j x^{\otimes j},\nonumber
  \end{align}
  where $K_j, H_j\in L^j_\Gamma(\ell^\infty(\mathbb{R}^n),
  \ell^\infty(\mathbb{R}^n))$. Taking derivatives on both sides of
  \begin{equation}
    F\circ K=K\circ H \label{Sternberg:NormalForms:Equation:Invariance}
  \end{equation}
  and evaluating at 0 we have
\[A\,K_1=K_1\,H_1.\] This equation has the obvious solution $K_1=\id$,
$H_1=A$, although other solutions are possible, for instance taking $K_1$ as any
linear map which commutes with $A$, like $K_1=\alpha\id$, $\alpha\in
\mathbb{R}$ and $H_1=A$. However, in what follows we make the choice $K_1=\id$ and $H_1=A$.

Taking $k$-th order derivatives on both sides of
\eqref{Sternberg:NormalForms:Equation:Invariance}, using the Fa\`a di Bruno
formula,
\begin{align}
  \sum_{j=1}^k\sum_{\substack{i_1,\ldots,
      i_j\geq1\\i_1+\ldots+i_j=k}}CD^jF&\circ K(D^{i_1}K\cdots
  D^{i_j}K) \nonumber\\ &=\sum_{j=1}^k\sum_{\substack{i_1,\ldots,
      i_j\geq1\\i_1+\ldots+i_j=k}}CD^jK\circ H(D^{i_1}H\cdots
  D^{i_j}H)\nonumber
\end{align}
(where for the sake of simplicity we have not written the dependence
of $C$ on the indices), and evaluating the derivatives at 0 we can
write
\begin{equation}
   AK_k+G^1_k=H_k+K_kA^{\otimes k}+G^2_k,\label{NormalForms:Theorem:Homological}
\end{equation}
where $G^1_k$, $G^2_k$ are $k$-linear maps which depend on $D^jF(0)$,
$2\leq j\leq r,$ and $K_i$, $H_i$, $2\leq i\leq r-1$.

Let $G_k=G_k^1-G_k^2$. Observe that $G_k$ consists of sums and contractions of
multilinear operators.

Using Sylvester operators (introduced in Section
\ref{Preliminaries:SylvesterInLk}) we rewrite Equation
\eqref{NormalForms:Theorem:Homological} as
\begin{equation}
  \left(\mathcal{S}_{A^{-1}, A^{\phantom{1}}}-\id\right)K_k=A^{-1}(-H_k+G_k).
  \label{NormalForms:Theorem:Homological_2}
\end{equation}

Now we proceed inductively from $k=2$ up to $k=r$. Assume that for $j$ up to the
$(k-1)$-th step we have obtained $K_j$ and $H_j$ in
$L^j_\Gamma(\ell^\infty(\mathbb{R}^n), \ell^\infty(\mathbb{R}^n))$ by solving
Equation \eqref{NormalForms:Theorem:Homological_2}. From the way $G_k$ is
defined, Proposition \ref{LandM:Proposition:LGAlgebraPropertyMultilinear} proves
that $G_k\in L_\Gamma^k(\ell^\infty(\mathbb{R}^n), \ell^\infty(\mathbb{R}^n))$
and hence \[A^{-1}(-H_k+G_k)\in L^k_\Gamma(\ell^\infty(\mathbb{R}^n),
\ell^\infty(\mathbb{R}^n)).\] Now by Proposition
\ref{Preliminaries:SylvesterGamma:Proposition}, if \[\spec_\Gamma(A)\cap
  \left(\spec_\Gamma(A)\right)^{k}=\emptyset,\] then $1\notin \spec(
\mathcal{S}_{A^{-1}, A^{\phantom{1}}})$, thus $\left(\mathcal{S}_{A^{-1},
    A}-\id\right):L^k_\Gamma\to L^k_\Gamma$ is invertible. This implies we can
choose $H_k=0$ and $K_k=(\mathcal{S}_{A^{-1},A}-\id)^{-1}A^{-1}G_k$.

Obviously, with this choice $H_k, K_k\in
L^k_\Gamma(\ell^\infty(\mathbb{R}^n), \ell^\infty(\mathbb{R}^n))$. On
the other hand, if
\[\spec_\Gamma(A)\cap \left(\spec_\Gamma(A)\right)^{k}\neq \emptyset\]
the operator $\mathcal{S}_{A^{-1}, A}$ may not be invertible and we
set $H_k=G_k$ and $K_k=0$. This is not the only possible choice, it is
only the simplest one and also has $K_k, H_k\in
L^k_\Gamma(\ell^\infty(\mathbb{R}^n), \ell^\infty(\mathbb{R}^n))$.

Another standard possibility is to decompose $L^k_\Gamma(E,
E)=\operatorname{Im}\mathcal{S}_{A^{-1}, A^{\phantom{1}}}\oplus V$,
where $\operatorname{Im}$ stands for the range of
$\mathcal{S}_{A^{-1}, A^{\phantom{1}}}$ and $V$ is a complementary
subspace in $L^k_\Gamma(E, E)$. Then one decomposes $A^{-1}G_k$
according to this splitting of the space as
$(A^{-1}G_k)^\mathrm{Im}+(A^{-1}G_k)^V$ and chooses $K_k$ such that
$(S_{A^{-1}, A^{\phantom{1}}}-\id )K_k=(A^{-1}G_k)^\mathrm{I}$ and
$H_k=A(A^{-1}G_k)^V$. Of course this choice also depends on the choice
of the complementary space $V$.

By the choices of $K_k$, $H_k,$ $1\leq k\leq r,$ 
\[D^k\left[F\circ K-K\circ H\right]=0,\qquad 1\leq k\leq r,\] and
hence, by Taylor's theorem, $F\circ K(x)-K\circ H(x)=o(\|x\|^r)$.

\end{proof}
\section{Sternberg theorems in lattices}\label{sec:sternberg-thms}

In this section we will prove several Sternberg conjugation theorems for
contractions under several non-resonance hypotheses. All of them are adaptations
to our setting of the classical proof in \cite{SternbergI}, using the normal
form theory developed in the previous section.

We will begin by proving Theorem \ref{Sternberg:Theorem:Sternberg}. First, two
remarks from the requirements.

\begin{remark}
  Note that we do not require $F_1$ to be small but only
  $B=DF_1(0)$ to be small in the $\Gamma$-norm.
\end{remark}
\begin{remark}
  Since $r_0$ is finite, assumption
  \ref{Sternberg:Theorem:Sternberg:H2} involves only a finite set
  of conditions.
\end{remark}

Before starting the proof we perform a rescaling of $F$ in order to transfer the
smallness conditions on the domain of definition to smallness of an auxiliary
parameter.

Let $\delta\in\mathbb{R}$, $\delta>0$, and the rescaling map $T_\delta x=\delta x$. We define
\[
F_\delta(x)=T^{-1}_\delta\circ F\circ
T_\delta(x)=Mx+N_\delta(x),
\]
where $N_\delta(x)=\delta^{-1}N(\delta x)$.

From now on we will not write the dependence on $\delta$ of
$F_\delta(x)$ and $N_\delta(x)$ and we will assume that $F$ is defined
on $B(0,1)\subset \ell^\infty(\mathbb{R}^n)$ and $\delta$ is as small
as needed. In particular, if $F$ is at least of class $C^2$, we have
that
\[
\|N\|_{C^0}=\mathcal{O}(\delta),\quad \|DN\|_{C^0}=\mathcal{O}(\delta)\quad\textrm{ and }\quad\|D^jN\|_{C^0}=\mathcal{O}(\delta^{j-1}),\qquad j\geq2,
\]
and moreover
\[
\|N\|_{C^r_\Gamma}=\mathcal{O}(\delta).
\]

Given $r, r_0\in\mathbb{N}$, $r\geq r_0$, we introduce the spaces
\begin{align}
  \chi^{r,r_0}&=\left\{g\in C^r(B(0,1),\ell^\infty(\mathbb{R}^n))\,\vert\, D^jg(0)=0,\,0\leq j\leq r_0,\, \|g\|_{C^r}<\infty\right\},\nonumber\\
  \chi^{r,r_0}_\Gamma&=\left\{g\in
    C_\Gamma^r(B(0,1),\ell^\infty(\mathbb{R}^n))\cap\chi^{r,r_0}\,\vert\,
    \|g\|_{C^r_\Gamma}<\infty\right\}.\nonumber
\end{align}
Observe that $\chi^{r,r_0}_\Gamma$ is a closed subspace of $C^r_\Gamma$.

\begin{lemma}\label{Sternberg:Lemma:OperatorWD}
Assume the hypotheses of Theorem \ref{Sternberg:Theorem:Sternberg}. 

Then, if $r\geq r_0$, $F\in C^r_\Gamma(U,
\ell^\infty(\mathbb{R}^n))$ and $\|B\|_\Gamma$ and the rescaling
parameter $\delta$ are small enough, the linear operator
$\mathcal{G}_m:\chi_\Gamma^{r,r_0}\to\chi_\Gamma^{r,r_0}$ defined by
\[\mathcal{G}_m(g)=M^{-m}g\circ F^m\]
is well defined and is a contraction in the $C^r_\Gamma$-norm.
\end{lemma}

\begin{proof}
  First, we fix some quantities to be used throughout the proof. From
  the definition of $r_0$ and the fact that $\beta<1$ we have
\[\alpha^{-1}\beta^{r_0}<1.\]
Then there exists $m\in \mathbb{N}$ such that
\[\Gamma(0)^{-2}\left(\alpha^{-1}\beta^{r_0}\right)^{m}<1\]
and also there exist positive numbers $\varepsilon_1$,
$\varepsilon_2$ and $\varepsilon_3$ such that
\begin{equation}
  \Gamma(0)^{-1}(\alpha^{-1}+\varepsilon_1)^m\left[\Gamma(0)^{-1}\left((\beta+\varepsilon_1)^m+\varepsilon_2\right)^{r_0}+\varepsilon_3\right]<1.\label{Sternberg:Lemma:OperatorWD:Equation:Epsilons}
\end{equation}

Note that this condition requires $\beta+\varepsilon_1<1$. There exists a norm
in $\mathbb{R}^n$ such that
\[\|\mathfrak{a}\|<\beta+\frac{\varepsilon_1}{2},\qquad
\|\mathfrak{a}^{-1}\|<\alpha^{-1}+\frac{\varepsilon_1}{2},\] where
$\|\cdot\|$ is the associated operator norm. Clearly, 
\[\|\mathfrak{a}^m\|<\left(\beta+\frac{\varepsilon_1}{2}\right)^m,\qquad
\|\mathfrak{a}^{-m}\|<\left(\alpha^{-1}+\frac{\varepsilon_1}{2}\right)^m\] and
\[\|A^m\|_\Gamma<\Gamma(0)^{-1}\left(\beta+\frac{\varepsilon_1}{2}\right)^m,\qquad
\|A^{-m}\|_\Gamma\leq
\Gamma(0)^{-1}\left(\alpha^{-1}+\frac{\varepsilon_1}{2}\right)^m.\] Moreover,
\begin{align}
  \|M^m\|_\Gamma&=\|(A+B)^m\|_\Gamma\nonumber\\ &\leq
  \|A^m\|_\Gamma+\mathcal{O}(\|B\|_\Gamma)\leq
  \Gamma(0)^{-1}(\beta+\frac{\varepsilon_1}{2})^m+\Gamma(0)^{-1}m\frac{\varepsilon_1}{2}\beta^{m-1}\nonumber\\
  &<\Gamma(0)^{-1}(\beta+\varepsilon_1)^m\nonumber
 \end{align}
 if $\|B\|_\Gamma$ is small enough.

 In the same way, now using Proposition \ref{LandM:Lemma:DInvertibleEasyCase},
 \begin{align}
   \|M^{-m}\|_\Gamma&\leq\|A^{-m}\|_\Gamma+\mathcal{O}(\|B\|_\Gamma)\nonumber\\
   &\leq \Gamma(0)^{-1}(\alpha^{-1}+\frac{\varepsilon_1}{2})^m+\Gamma(0)^{-1}m\frac{\varepsilon_1}{2}\alpha^{-(m-1)}<\Gamma(0)^{-1}(\alpha^{-1}+\varepsilon_1)^m. \nonumber
 \end{align}
Analogously,
\[\|M^m\|\leq \|M\|^m\leq (\beta+\varepsilon_1)^m.\]

Let $g\in\chi^{r,r_0}_\Gamma$. By Remark
\ref{DiffFunDecay:Remark:CompositionLinearMap} we have 
\[\|\mathcal{G}_m(g)\|_{C^r_\Gamma}\leq \|M^{-m}\|_\Gamma\|g\circ
  F^m\|_{C^r_\Gamma}.\] To estimate $\|g\circ F^m\|_{C^r_\Gamma}$ we will use
the Fa\`a di Bruno formula for the $p$-th derivative of $g\circ F^m$, $1\leq
p\leq r$,
\begin{align}
   D^p(g\circ F^m)(x)=&D^pg(F^m(x))(DF^m(x))^{\otimes p}\nonumber\\ &+\sum_{j=1}^{p-1}\sum_{\substack{i_1,\ldots,i_j\geq1\\i_1+\ldots+i_j=p}}CD^jg(F^m(x))D^{i_1}F^m(x)\cdots D^{i_j}F^m(x),\label{Sternberg:Lemma:OperatorWD:Faa}
\end{align}
where $C$ is a combinatorial coefficient which depends on all
indices in the sum. From \eqref{Sternberg:Lemma:OperatorWD:Faa} it is clear that
$D^p(g\circ F^m)(0)=0$ for $1\leq p\leq r_0$, since $F(0)=0$.

Since $g\in \chi^{r,r_0}$, by Taylor's theorem in integral form (see
\cite{MarsdenRatiu}),
\[g(x)=\frac{1}{(r_0-1)!}\int_0^1(1-t)^{r_0-1}D^{r_0}g(tx)x^{\otimes r_0}\dd t\]
and also
\[D^jg(x)=\frac{1}{(r_0-j-1)!}\int_0^1(1-t)^{r_0-j-1}D^{r_0}g(tx)x^{\otimes (r_0-j)}\dd t,
\qquad 0\leq j \leq r_0-1.\]

Using the previous formulas, Proposition \ref{LandM:Proposition:ContractionNorm}
and usual results about integration in Banach spaces (see \cite{MarsdenRatiu}
for the theory of Cauchy-Bochner integration on Banach spaces) we have
\begin{align}
   \|D^j&g(F^m(x))\|_\Gamma\nonumber\\&\leq 
\frac{1}{(r_0-j-1)!}\int_0^1(1-t)^{r_0-j-1}\|D^{r_0}g(tF^m(x))\|_\Gamma\|F^m(x)\|^{ r_0-j}\dd t\nonumber\\ &\leq\frac{1}{(r_0-j)!}\|g\|_{C^r_\Gamma}\|F^m(x)\|^{r_0-j}, \qquad 0\leq j\leq r_0-1\nonumber
\end{align}
and 
\[
\|D^jg(F^m(x))\|_\Gamma\leq \|g\|_{C^r_\Gamma},\qquad r_0\leq j\leq r.
\]

As a consequence of the two previous bounds, we can write the more compact form
\[
\|D^jg(F^m(x))\|_\Gamma\leq
\|g\|_{C^r_\Gamma}\|F^m(x)\|^{(r_0-j)_+},\qquad 0\leq j\leq r,
\] where $(t)_+=\max(t,0)$.

Then, using Proposition \ref{LandM:Proposition:LGAlgebraPropertyMultilinear}
\begin{align*}
\|D^p(g&\circ F^m)(x) \|_\Gamma \\  \leq 
& \|g\|_{C^r_\Gamma}\|F^m(x)\|^{(r_0-p)_+}\|DF^m(x)\|_\Gamma\|DF^m(x)\|^{p-1} \\ &+\sum_{j=1}^{p-1}\sum_{\substack{i_1,\ldots,i_j\geq1\\i_1+\ldots+i_j=p}}
  C\|g\|_{C^r_\Gamma}\|F^m(x)\|^{(r_0-j)_+}\|D^{i_1}F^m(x)\|_\Gamma
  \cdots\|D^{i_j}F^m(x)\|_\Gamma.
\end{align*}
By the rescaling, if $x\in B(0,1)$,
\begin{align*}
  \|F^m(x)\|&\leq \|M^mx\|+\mathcal{O}(\delta)\leq \|M^m\|+\mathcal{O}(\delta),\\
  \|DF^m(x)\|&\leq \|M^m\|+\mathcal{O}(\delta),\\
  \|DF^m(x)\|_\Gamma&\leq \|M^m\|_\Gamma+\mathcal{O}(\delta)
\end{align*}
and
\[\|D^j F^m(x)\|_\Gamma=\mathcal{O}(\delta),\qquad j\geq 2.\]
 Also note
that for $p\geq 0$, we have $(r_0-p)_++p\geq r_0$.

Then
\[\|D^p(g\circ F^m)(x)\|_\Gamma\leq \|g\|_{C^r_\Gamma}\left[\big(\|M^m\|_\Gamma+\mathcal{O}(\delta)\big)\big(\|M^{m}\|+\mathcal{O}(\delta)\big)^{r_0-1}+\mathcal{O}(\delta)\right],\]
for $1\leq p\leq r$, and finally,
\begin{align}
   \|\mathcal{G}_m(g)\|_{C^r_\Gamma}&\leq \|M^{-m}\|_\Gamma\left[\left(\|M^m\|_\Gamma+\mathcal{O}(\delta)\right)\left(\|M^{m}\|+\mathcal{O}(\delta)\right)^{r_0-1}+\mathcal{O}(\delta)\right]\|g\|_{C^r_\Gamma}\nonumber\\
   &\leq \Gamma(0)^{-1}(\alpha^{-1}+\varepsilon_1)^m\nonumber\\ &\phantom{\leq\ }\times\left[\left(\Gamma(0)^{-1}(\beta+\varepsilon_1)^m+\mathcal{O}(\delta)\right)\left((\beta+\varepsilon_1)^m+\mathcal{O}(\delta)\right)^{r_0-1}+\mathcal{O}(\delta)\right]\|g\|_{C^r_\Gamma}.\nonumber
\end{align}

Then if $\delta$ is small enough, by
\eqref{Sternberg:Lemma:OperatorWD:Equation:Epsilons} the factor in
front of $\|g\|_{C^r_\Gamma}$ is strictly less than 1 and hence
$\mathcal{G}$ is a contraction in $\chi^{r, r_0}_\Gamma$.

\end{proof}

Now we use the normal form theory in the previous section to find a
decay map which linearises our map $F$ up to order $r_0$. The form of
$A$ implies that $\spec (A)=\{\lambda_1, \ldots, \lambda_n\}$ and $\spec
(A^m)=\{\lambda_1^m,\ldots, \lambda_n^m\}$. Since $A^m$ is uncoupled,
$\spec_\Gamma(A^m)=\spec(A^m)$. Moreover the non-resonance condition
\ref{Sternberg:Theorem:Sternberg:H2} implies that
\[\lambda_i^m\neq \lambda^{mk_1}_1\cdots \lambda_n^{mk_n}, \qquad
k\in(\mathbb{Z}^+)^n,\quad 2\leq |k| \leq r_0, \]
and therefore 
\[(\spec_\Gamma(A^m))^j\cap\spec_\Gamma(A^m)=\emptyset, \qquad
j\geq2.\]

Taking $\|B\|_\Gamma$ sufficiently small, since $\spec_\Gamma$ is upper
semicontinuous by Proposition \ref{Preliminaries:Operational:Prop:Continuity}, 
we have
\[(\spec_\Gamma (M^m))^j\cap \spec_\Gamma (M^m)=\emptyset, \qquad 2\leq j\leq
  r_0,\] because we are only dealing with a finite set of conditions.

Hence Corollary \ref{Sternberg:NormalForms:Corollary} gives us that
there exists a polynomial $K\in
C^\infty_\Gamma(\ell^\infty(\mathbb{R}^n), \ell^\infty(\mathbb{R}^n))$
of degree (at most) $r_0$ such that $K(0)=0$, $DK(0)=\id$ and
\[
F^m\circ K(x)-K\circ M^m(x)=o(\|x\|^{r_0}).
\]
Let $S_0=K^{-1}$ be the local inverse. Taking the rescaling parameter $\delta$
smaller if necessary we can assume that $S_0$ is defined in
$B(0,1)\subset\ell^\infty(\mathbb{R}^n)$. By Theorem
\ref{DiffFunDecay:InverseFunctionTheorem}, we have $S_0\in C^r_\Gamma$ and satisfies
\begin{align}
  S_0(0)=0,\qquad DS_0(0)=\id,\nonumber\\
M^{-m}\circ S_0\circ F^m-S_0=o(\|x\|^{r_0}).\nonumber
\end{align}
Starting with this approximate conjugation we define the sequence
\begin{equation}
  S_n=M^{-m}S_{n-1}\circ F^m=\mathcal{G}_m(S_{n-1}),\qquad n\geq1.\nonumber
\end{equation}

The next lemma proves that $S_n$ converges to a well-defined conjugation in the space $C^r_\Gamma$.

\begin{lemma}\label{Sternberg:Lemma:RmGammaConverges}
  The sequence $\{S_n\}_{n\in\mathbb{N}}$ defined above converges to
  a function $S\in C^r_\Gamma(B(0,1),\ell^\infty(\mathbb{R}^n))$ satisfying
  $S(0)=0$, $DS(0)=\id$ and
  \[S \circ F^m= M^m S.\]
\end{lemma}
\begin{proof}
  Since $m$ is fixed we will drop the dependence of $\mathcal{G}_m$
  on $m$. First, we prove the following relation
  \begin{align}
    S_n=S_0+\sum_{j=0}^{n-1}\mathcal{G}^j(M^{-m}S_0\circ F^m-S_0),\label{SternbergContractions:Lemma:RmGammaConverges:Eq1}
  \end{align}
  where $\mathcal{G}^0=\id$ and $\mathcal{G}^j=\mathcal{G}\circ
  \mathcal{G}^{j-1},\, j\geq 1$.  Observe that $M^{-m}S_0\circ
  F^m-S_0\in \chi^{r, r_0}$, since $S_0$ solves the conjugation
  equation formally up to order $r_0$. Moreover $M^{-m}S_0\circ
  F^m-S_0\in C^r_\Gamma$ since $M\in L_\Gamma$ and $S_0, F\in C^r_\Gamma$. We 
  prove \eqref{SternbergContractions:Lemma:RmGammaConverges:Eq1} by
  induction. When $n=1$, we use the definition
  $S_1=\mathcal{G}(S_0)$:
  \begin{align}
    S_1=\mathcal{G}(S_0)=M^{-m}S_0\circ
    F^m=S_0+\mathcal{G}^0(M^{-m}S_0\circ F^m-S_0). \nonumber
  \end{align}
  Now assume Equation \eqref{SternbergContractions:Lemma:RmGammaConverges:Eq1}
  is true up to index $n$, then
  \begin{align}
    S_{n+1}&=M^{m(-n-1)}S_0\circ F^{m(n+1)}\nonumber\\
    &=M^{-mn}S_0\circ F^{mn}+M^{m(-n-1)}S_0\circ
    F^{m(n+1)}-M^{-mn}S_0\circ F^{mn}\nonumber\\
    &=S_n+M^{-mn}\left(M^{-m}S_0\circ F^{m}-S_0\right)\circ
    F^{mn}\nonumber\\
    &=S_n+\mathcal{G}^n(M^{-m}S_0\circ F^m-S_0)\nonumber\\
    &=S_0+\sum_{j=0}^{n}\mathcal{G}^j\big(M^{-m}S_0\circ
    F^m-S_0\big)\nonumber.
  \end{align}

  By Lemma \ref{Sternberg:Lemma:OperatorWD}, $\mathcal{G}$ is a contraction in
  $\chi^{r, r_0}_\Gamma$ and therefore the series arising from
  \eqref{SternbergContractions:Lemma:RmGammaConverges:Eq1} converges and
  $\lim_{n\to\infty} S_n$ exists and belongs to $\chi^{r, r_0}_\Gamma$.

  Finally, we check the conjugacy property. Indeed,
  \begin{align*}
    S\circ F^m& =\lim_{n\to\infty}S_n\circ F^m=\lim_{n\to\infty}
    M^{-mn}S_0\circ F^{mn+m}\\ &=\lim_{n\to\infty}M^mM^{-mn-m}S_0F^{mn+m}=M^mS.
  \end{align*}
Also
  \begin{align}
    S(0)=\lim_{n\to\infty}M^{-mn}S_0\circ F^{mn}(0)=0,\nonumber
  \end{align}
  and since $DF^{mn}(0)=DF(F^{mn-1}(0))\cdots DF(0)=M^{mn}$ and
  $DS_0(0)=\id$,
  \begin{align}
    DS(0)=\lim_{n\to\infty}M^{-mn}DS_0(F^{mn}(0))DF^{mn}(0)=\lim_{n\to\infty}M^{-mn}\id
    M^{mn}=\id.\nonumber
  \end{align}
\end{proof}

Thus, $S$ conjugates $F^m$ to $M^m$. The final step is to show that $S$
also conjugates $F$ to $M$.

By the spectral properties and Corollary
\ref{Sternberg:NormalForms:Corollary}, there exists a polynomial
$\wt{K}\in C^\infty(\ell^\infty(\mathbb{R}^n), \ell^\infty(\mathbb{R}^n))$ such that
\[\wt{K}(0)=0, \qquad D\wt{K}(0)=\id,\]
and
\[F\circ\wt{K}(x)-\wt{K}\circ M(x)=o(\|x\|^{r_0}).\]

Let $R_0=\wt{K}^{-1}$, which similarly to $S_0$, we can assume is defined in
$B(0,1)\subset \ell^\infty(\mathbb{R}^n)$. Thus
\[M^{-1}R_0\circ F(x)=R_0(x)+o(\|x\|^{r_0})\]
and as a consequence
\begin{equation}
   M^{-m}R_0\circ F^m(x)=R_0(x)+o(\|x\|^{r_0}).\label{Sternberg:Theorem:Unicity:Eq1}
\end{equation}

\begin{lemma}
  Under the hypotheses of Theorem \ref{Sternberg:Theorem:Sternberg},
  if $\|B\|$ and the rescaling parameter $\delta$ are small enough, the operator
  \[\wt{\mathcal{G}}:C^r(B(0,1),\ell^\infty(\mathbb{R}^n))\to
  C^r(B(0,1),\ell^\infty(\mathbb{R}^n))\] defined by
\[\wt{\mathcal{G}}(g)=M^{-1}g\circ F\]
is well defined and is a contraction in the $C^r$-norm.
\end{lemma}
We omit the proof of this lemma since it is completely analogous to the proof of Lemma
\ref{Sternberg:Lemma:OperatorWD} but the estimates are much simpler,
since they do not involve decay functions.

We define the sequence 
\[ 
R_n=M^{-1}R_{n-1}\circ F, \qquad n\geq 1.
\] 
The same arguments as the
ones used in the proof of Lemma \ref{Sternberg:Lemma:RmGammaConverges}
but now in the space $C^r$ instead of $C^r_\Gamma$ give that there
exists $R=\lim_{n\to\infty}R_n$ with  $R\in
C^r(B(0,1), \ell^\infty(\mathbb{R}^n))$ such that
$R\circ F=MR$.

Now consider the iteration $S_n$ introduced in the proof of Lemma
\ref{Sternberg:Lemma:RmGammaConverges} with $S_0=R_0\in
C^r_\Gamma$. Since $R_0$ satisfies
\eqref{Sternberg:Theorem:Unicity:Eq1}, we see that $S_n$ is a
subsequence of $R_n$, being both sequences convergent in the larger space
$C^r$. Then \[R=\lim_{n\to\infty}R_n=\lim_{n\to\infty}S_n=S\in
C^r_\Gamma\] and therefore $S$ also conjugates $F$ with $M$, proving Theorem
\ref{Sternberg:Theorem:Sternberg}.

An improvement of Theorem \ref{Sternberg:Theorem:Sternberg} consists of
not assuming the non-resonance condition
\ref{Sternberg:Theorem:Sternberg:H2}. In such case we obtain a
$C^r_\Gamma$ local conjugation to a normal form of $F$ instead to a
conjugation to its linear part.

\begin{theorem}\label{Sternberg:Theorem:SternbergNR}
  Under the conditions and notation of Theorem
  \ref{Sternberg:Theorem:Sternberg} except hypothesis
  \ref{Sternberg:Theorem:Sternberg:H2}, if $F\in
  C^r_\Gamma(\ell^\infty(\mathbb{R}^n), \ell^\infty(\mathbb{R}^n))$
  with $r\geq r_0$ and $\|B\|_\Gamma$ is small enough there exists a
  polynomial $H\in C^\infty_\Gamma(\ell^\infty(\mathbb{R}^n),
  \ell^\infty(\mathbb{R}^n))$ of degree not larger than $r_0$ and
  $R\in C^r_\Gamma(\ell^\infty(\mathbb{R}^n),
  \ell^\infty(\mathbb{R}^n))$ such that
  \[R(0)=0,\qquad DR(0)=\id\]
  and 
  \[R\circ F=H\circ R\] in some neighborhood $U_1\subset U$ of 0.
\end{theorem}

\begin{proof}
  We will only comment on the differences of this proof with the proof
  of Theorem \ref{Sternberg:Theorem:Sternberg}. We rescale the map, we consider
  the spaces $\chi^{r, r_0}$ and $\chi^{r,r_0}_\Gamma$ and we use the same
  integer $m$ as in the proof of that theorem. We take a normal form $H$ provided by Theorem
  \ref{Sternberg:NormalForms:Theorem} and define the operator
  \[\overline{\mathcal{G}}_m(g)=H^{-m}\circ g\circ F^m.\]
  Now the estimates on $\overline{\mathcal{G}}_m$ become more involved
  because in this case $H$ is not linear. This implies that
  $\overline{\mathcal{G}}_m$ is not linear anymore. However, because
  of the rescaling, $H^{-1}$ is very close to $M^{-1}$ in
  $C^r_\Gamma$ (and $C^r$) norm, a fact which gives similar estimates
  and thus proves $\Lip \overline{\mathcal{G}}_m<1$. The remaining
  part of the proof is analogous.

\end{proof}

The previous theorems assume that the linear part of the maps is close
to an uncoupled map with identical dynamics on each node. This
gives sufficient conditions for the conjugation in terms of the
eigenvalues of the projections to the nodes.

Theorem \ref{Sternberg:Theorem:SternbergSpectral} requires instead conditions on
the $\Gamma$-spectrum of the linear part of the map.

Note that, since $\spec_\Gamma(A)$ is compact, Hypothesis
\ref{Sternberg:Theorem:SternbergSpectral:H1} in Theorem
\ref{Sternberg:Theorem:SternbergSpectral} implies that
$0<\alpha_{\Gamma}\leq\beta_{\Gamma}<1$ and $r_0<\infty$.

\begin{proof}[Proof of Theorem \ref{Sternberg:Theorem:SternbergSpectral}]

  The structure of the proof is very similar to the one of the proof of Theorem
  \ref{Sternberg:Theorem:Sternberg} but it has some technical
  differences. Let $r_0$ and $r$ be as in the statement of the
  theorem. Note that $\beta_{\Gamma}=r_\Gamma(A)$ and
  $\alpha_{\Gamma}^{-1}=r_\Gamma(A^{-1})$. Since $r_0>\nu$ then
  $\alpha_{\Gamma}^{-1}\beta_{\Gamma}^{r_0}<1$ and there exists $\varepsilon_1>0$ such
  that
\[(\alpha_{\Gamma}^{-1}+\varepsilon_1)(\beta_{\Gamma}+\varepsilon_1)^{r_0}<1.\]

In any Banach algebra
\begin{equation}
  r_\Gamma(A)=\lim_{n\to\infty}\left(\|A^n\|_\Gamma\right)^{1/n}=\inf_{n\geq1}\left(\|A^n\|_\Gamma\right)^{\frac{1}{n}},
\end{equation}
thus there exists $m\in \mathbb{N}$ such that
\[\|A^n\|_\Gamma\leq\left(r_\Gamma(A)+\varepsilon_1\right)^n,\qquad
n\geq m,\] and
\[\|A^{-n}\|_\Gamma\leq(r_\Gamma(A^{-1})+\varepsilon_1)^n,\qquad n\geq m.\]
Obviously,
\[(\alpha_{\Gamma}^{-1}+\varepsilon_1)^m(\beta_{\Gamma}+\varepsilon_1)^{mr_0}<1\] and
there exists $\varepsilon_2, \varepsilon_3>0$ such that
\[(\alpha_{\Gamma}^{-1}+\varepsilon_1)^m\left[\left((\beta_{\Gamma}+\varepsilon_1)^m+\varepsilon_2\right)^{r_0}+\varepsilon_3\right]<1.\]

Now we introduce the operator $\mathcal{G}_m:\chi^{r,r_0}_\Gamma\to
\chi_\Gamma^{r,r_0}$ defined by
\[\mathcal{G}_m(g)=A^{-m}g\circ F^m.\]

Analogous estimates as in Lemma \ref{Sternberg:Lemma:OperatorWD} yield
that if the rescaling parameter is small enough,
$\mathcal{G}_m$ is well defined in $\chi_\Gamma^{r,r_0}$ and is a
contraction. Then the proof follows the same lines as the proof of
Theorem \ref{Sternberg:Theorem:Sternberg}.

For the uniqueness arguments needed at the end of the proof, we
consider the operator $\wt{\mathcal{G}}(g)=A^{-1}g\circ F$ in
$C^r(B(0,1), \ell^\infty(\mathbb{R}^n))$. Since $\spec(A)\subset
\spec_\Gamma(A)$, condition
\ref{Sternberg:Theorem:SternbergSpectral:H2} implies that there are
also no resonances among the elements of $\spec(A)$ and that
$r(A)<\beta_{\Gamma}$ and $r(A^{-1})<\alpha_{\Gamma}^{-1}$. Hence we can find a norm in
the space $\ell^\infty(\mathbb{R}^n)$, equivalent to the original one, such that
\begin{equation}
   \|A^{-1}\|\|A\|^{r_0}<1,\label{Sternberg:Theorem:SternbergSpectral:Cr:Eq1}
\end{equation}
where in the previous expression $\|\cdot\|$ stands for the corresponding
operator norm. The bound
\eqref{Sternberg:Theorem:SternbergSpectral:Cr:Eq1} allows us to prove
the estimates needed to show that $\wt{\mathcal{G}}$ is a
contraction. With these ingredients we can finish the proof in this
setting in the same way as in Theorem \ref{Sternberg:Theorem:Sternberg}.

\end{proof}

The analogous version of Theorem \ref{Sternberg:Theorem:SternbergNR}
in this setting is the following.

\begin{theorem}\label{Sternberg:Theorem:SternbergSpectralNR}
  Under the conditions of Theorem
  \ref{Sternberg:Theorem:SternbergSpectral}, except condition
  \ref{Sternberg:Theorem:SternbergSpectral:H2}, if $F\in
  C^r_\Gamma(\ell^\infty(\mathbb{R}^n), \ell^\infty(\mathbb{R}^n))$
  with $r\geq r_0$ there exists a polynomial $H\in
  C^\infty_\Gamma(\ell^\infty(\mathbb{R}^n),
  \ell^\infty(\mathbb{R}^n))$ of degree not larger than $r_0$ and
  $R\in C^r_\Gamma(\ell^\infty(\mathbb{R}^n),
  \ell^\infty(\mathbb{R}^n))$ such that
  \[R(0)=0,\quad DR(0)=\id\]
  and 
  \[R\circ F=H\circ R\] in some neighborhood $U_1\subset U$ of 0.

\end{theorem}

The proof of this theorem is a combination of the arguments in the proofs of
Theorem \ref{Sternberg:Theorem:SternbergNR} and Theorem
\ref{Sternberg:Theorem:SternbergSpectral}.


\begin{acknowledgements}
The authors acknowledge the support of the Spanish grant Mineco MTM2013--41168--P
and the Catalan grant AGAUR 2014 SGR 1145.
E.F. also  acknowledges the support of  MTM2016--80117--P (MINECO/FEDER, UE).
\end{acknowledgements}

\bibliographystyle{spmpsci}
\bibliography{}

\def\cprime{$'$}
\begin{thebibliography}{10}
\providecommand{\url}[1]{{#1}}
\providecommand{\urlprefix}{URL }
\expandafter\ifx\csname urlstyle\endcsname\relax
  \providecommand{\doi}[1]{DOI~\discretionary{}{}{}#1}\else
  \providecommand{\doi}{DOI~\discretionary{}{}{}\begingroup
  \urlstyle{rm}\Url}\fi

\bibitem{MarsdenRatiu}
Abraham, R., Marsden, J.E., Ratiu, T.S.: Manifolds, Tensor Analysis, and
  Applications, 2nd Ed.
\newblock Springer-Verlag, Berlin (1988)

\bibitem{Afra05}
Afraimovich, V.: Some topological properties of lattice dynamical systems.
\newblock In: Dynamics of coupled map lattices and of related spatially
  extended systems, \emph{Lecture Notes in Phys.}, vol. 671, pp. 153--179.
  Springer, Berlin (2005)

\bibitem{BanyagaLlaveWayne}
Banyaga, A., de~la Llave, R., Wayne, C.E.: Cohomology equations near hyperbolic
  points and geometric versions of {S}ternberg linearization theorem.
\newblock The Journal of Geometric Analysis \textbf{6}(4), 613--649 (1996)

\bibitem{FrenkelKontorova}
Braun, O.M., Kivshar, Y.S.: The {F}renkel-{K}ontorova model: {C}oncepts,
  methods, and applications.
\newblock Texts and Monographs in Physics. Springer-Verlag, Berlin (2004)

\bibitem{B-S88}
Bunimovich, L.A., Sina\u\i, Y.G.: Spacetime chaos in coupled map lattices.
\newblock Nonlinearity \textbf{1}(4), 491--516 (1988)

\bibitem{CabreFontichDeLaLlavePMIM1}
Cabr{\'e}, X., Fontich, E., de~la Llave, R.: The parameterization method for
  invariant manifolds. {I}. {M}anifolds associated to non-resonant subspaces.
\newblock Indiana Univ. Math. J. \textbf{52}(2), 283--328 (2003).
\newblock \doi{10.1512/iumj.2003.52.2245}

\bibitem{ChaperonCoudrayInvariantManifolds}
Chaperon, M., Coudray, F.: Invariant manifolds, conjugacies and blow-up.
\newblock Ergodic Theory Dynam. Systems \textbf{17}(4), 783--791 (1997).
\newblock \doi{10.1017/S0143385797085052}

\bibitem{dauxois2002dynamics}
Dauxois, T., Ruffo, S., Arimondo, E., Wilkens, M.: Dynamics and thermodynamics
  of systems with long-range interactions: An introduction.
\newblock Lecture Notes in Physics pp. 1--22 (2002)

\bibitem{Dehlinger1929}
Dehlinger, U.: Zur {T}heorie der rekristallisation reiner {M}etalle.
\newblock Annalen der Physik \textbf{394}(7), 749--793 (1929)

\bibitem{Ermentrout2010}
Ermentrout, G.B., Terman, D.H.: Mathematical foundations of neuroscience,
  \emph{Interdisciplinary Applied Mathematics}, vol.~35.
\newblock Springer, New York (2010).
\newblock \doi{10.1007/978-0-387-87708-2}

\bibitem{FermiPU}
Fermi, E., Pasta, J., Ulam, S.: Studies on nonlinear problems.
\newblock Document LA 1940  (1955)

\bibitem{floria2005frenkel}
Flor{\'\i}a, L.M., Baesens, C., G{\'o}mez-Garde{\~n}es, J.: The
  {F}renkel--{K}ontorova model.
\newblock Dynamics of coupled map lattices and of related spatially extended
  systems pp. 209--240 (2005)

\bibitem{FontichDelaLlaveMartinFAFramework}
Fontich, E., de~la Llave, R., Mart{\'{\i}}n, P.: Dynamical systems on lattices
  with decaying interaction {I}: a functional analysis framework.
\newblock J. Differential Equations \textbf{250}(6), 2838--2886 (2011).
\newblock \doi{10.1016/j.jde.2010.07.023}

\bibitem{FontichDelaLlaveMartinHSIM}
Fontich, E., de~la Llave, R., Mart{\'{\i}}n, P.: Dynamical systems on lattices
  with decaying interaction {II}: hyperbolic sets and their invariant
  manifolds.
\newblock J. Differential Equations \textbf{250}(6), 2887--2926 (2011).
\newblock \doi{10.1016/j.jde.2011.01.015}

\bibitem{FontichDeLaLlaveSireWhiskeredLattices15}
Fontich, E., de~la Llave, R., Sire, Y.: Construction of invariant whiskered
  tori by a parameterization method. {II}: Quasi-periodic and almost periodic
  breathers in coupled map lattices.
\newblock J. of Differential Equations \textbf{259}(6), 2180--2279 (2015)

\bibitem{Kontorova1938_1}
Frenkel, Y.I., Kontorova, T.A.: The model of dislocation in solid body.
\newblock Zh. Eksp. Teor. Fiz \textbf{8}(1340) (1938)

\bibitem{Frenkel1939}
Frenkel, Y.I., Kontorova, T.A.: On the theory of plastic deformation and
  twinning.
\newblock Izv. Akad. Nauk, Ser. Fiz. \textbf{1}, 137--149 (1939)

\bibitem{FP99}
Friesecke, G., Pego, R.L.: Solitary waves on {FPU} lattices. {I}. {Q}ualitative
  properties, renormalization and continuum limit.
\newblock Nonlinearity \textbf{12}(6), 1601--1627 (1999).
\newblock \doi{10.1088/0951-7715/12/6/311}

\bibitem{GallavottiFPU}
Gallavotti, G.: Introduction to the {F}ermi-{P}asta-{U}lam problem.
\newblock In: The Fermi-Pasta-Ulam Problem: A Status Report, Lecture Notes in
  Physics, pp. 1--8. Springer Berlin Heidelberg (2007)

\bibitem{IlyashenkoYakovenko}
{Il'yashenko}, Y.S., {Yakovenko}, S.Y.: {Finitely-smooth normal forms of local
  families of diffeomorphisms and vector fields.}
\newblock {Russ. Math. Surv.} \textbf{46}(1), 1--43 (1991).
\newblock \doi{10.1070/RM1991v046n01ABEH002733}

\bibitem{Izhikevich07}
Izhikevich, E.M.: Dynamical systems in neuroscience: the geometry of
  excitability and bursting.
\newblock Computational Neuroscience Series. MIT Press, Cambridge, MA (2007)

\bibitem{JiangDeLaLlaveSRBMeasures}
Jiang, M., de~la Llave, R.: Smooth dependence of thermodynamic limits of
  {SRB}-measures.
\newblock Comm. Math. Phys. \textbf{211}(2), 303--333 (2000)

\bibitem{jiang-pesin}
Jiang, M., Pesin, Y.B.: Equilibrium measures for coupled map lattices:
  existence, uniqueness and finite-dimensional approximations.
\newblock Comm. Math. Phys. \textbf{193}(3), 675--711 (1998).
\newblock \doi{10.1007/s002200050344}

\bibitem{BEJohnsonBanach}
Johnson, B.E.: Banach algebras: introductory course.
\newblock In: Algebras in analysis ({P}roc. {I}nstructional {C}onf. and {NATO}
  {A}dvanced {S}tudy {I}nst., {B}irmingham, 1973), pp. 63--83. Academic Press,
  London (1975)

\bibitem{kaneko1993theory}
Kaneko, K.: Theory and applications of coupled map lattices.
\newblock John Wiley \& Son Ltd, Chichester (1993)

\bibitem{Pesin-Yurchenko}
Pesin, Y.B., Yurchenko, A.A.: Some physical models described by the
  reaction-diffusion equation, and coupled map lattices.
\newblock Uspekhi Mat. Nauk \textbf{59}(3(357)), 81--114 (2004).
\newblock \doi{10.1070/RM2004v059n03ABEH000737}

\bibitem{peyrard2004}
Peyrard, M.: Nonlinear dynamics and statistical physics of {DNA}.
\newblock Nonlinearity \textbf{17}(2), R1--R40 (2004).
\newblock \doi{10.1088/0951-7715/17/2/R01}

\bibitem{peyrard2004breathers}
Peyrard, M., Sire, Y.: Breathers in biomolecules?
\newblock Energy Localisation and Transfer pp. 325--340 (2004)

\bibitem{Prandtl1928}
Prandtl, L.: A conceptual model to the kinetic theory of solid bodies.
\newblock Z. Angew. Math. Mech \textbf{8}, 85--106 (1928)

\bibitem{RudinFA}
Rudin, W.: Functional analysis, second edn.
\newblock International Series in Pure and Applied Mathematics. McGraw-Hill
  Inc., New York (1991)

\bibitem{SchechterFunctionalAnalysis}
Schechter, M.: Principles of functional analysis, vol.~2.
\newblock Academic Press New York (1971)

\bibitem{SternbergI}
Sternberg, S.: Local contractions and a theorem of {P}oincar\'e.
\newblock Amer. J. Math. \textbf{79}, 809--824 (1957)

\bibitem{stern58}
Sternberg, S.: On the structure of local homeomorphisms of {E}uclidean
  {$n$}-space. {II}.
\newblock Amer. J. Math. \textbf{80}, 623--631 (1958).
\newblock \doi{10.2307/2372774}

\end{thebibliography}
\end{document}